\documentclass[a4paper, 12pt, onecolumn]{article} \textwidth 148mm
\usepackage{amsmath}
\usepackage{amsthm}
\usepackage{amsfonts}
\usepackage{bbm}
\usepackage{CJK}
\usepackage{fancyhdr}
\usepackage{graphicx}
\usepackage{geometry}
\usepackage{psfrag}
\usepackage{amsfonts,amsmath,amsthm, amssymb}
\usepackage{latexsym, euscript, epic, eepic}
\usepackage{time}
\usepackage{txfonts}
\usepackage{colortbl}
\usepackage{stmaryrd}
\usepackage{mathrsfs}
\usepackage{txfonts}
\usepackage{amsfonts}
\usepackage{color}
\usepackage{lineno}

\usepackage{indentfirst,latexsym,bm}

 \numberwithin{equation}{section}
\begin{document}
\newtheorem{theorem}{Theorem}[section]
\newtheorem{proposition}[theorem]{Proposition}
\newtheorem{remark}[theorem]{Remark}
\newtheorem{corollary}[theorem]{Corollary}
\newtheorem{definition}{Definition}[section]
\newtheorem{lemma}[theorem]{Lemma}

\title{\large Measure Upper Bounds for Nodal Sets of Eigenfunctions of the bi-Harmonic Operator}
\author{\normalsize Long Tian \\
\scriptsize Department of Applied Mathematics, Nanjing University
of Science $\&$ Technology, Nanjing, China
\\
\scriptsize email: tianlong19850812@163.com\\
\normalsize Xiaoping Yang\\
\scriptsize Department of Mathematics, Nanjing University, Nanjing, China
\\
\scriptsize email: xpyang@nju.edu.cn}

\date{}
\maketitle
 \fontsize{12}{22}\selectfont\small
 \paragraph{Abstract:}
 In this article, we consider eigenfunctions $u$ of the bi-harmonic operator, i.e.,
 $\triangle^2u=\lambda^2u$ on $\Omega$ with some homogeneous linear boundary conditions.
We assume that $\Omega\subseteq\mathbb{R}^n$ ($n\geq2$) is a $C^{\infty}$ bounded domain, $\partial\Omega$ is piecewise analytic and $\partial\Omega$ is analytic except a set $\Gamma\subseteq\partial\Omega$ which is a finite union of some compact $(n-2)$ dimensional submanifolds of $\partial\Omega$. The main result of this paper is that the measure upper bounds of the nodal sets of the eigenfunctions is controlled by $\sqrt{\lambda}$. We first define a frequency function and a doubling index related to these eigenfunctions. With the help of establishing the monotonicity formula, doubling conditions and various a priori estimates, we obtain that the $(n-1)$ dimensional Hausdorff measures of nodal sets of these eigenfunctions in a ball are controlled by the frequency function and $\sqrt{\lambda}$.
In order to further control the frequency function with $\sqrt{\lambda}$, we first establish the relationship between the frequency function and the doubling index, and then separate the domain $\Omega$ into two parts: a domain away from $\Gamma$ and a domain near $\Gamma$, and develop iteration arguments to deal with the two cases respectively.
\\[10pt]
\\[10pt]

\vspace{1cm}\fontsize{12}{22}\selectfont
\footnote{This work is supported by National Science Foundation of China (No.11401307,No.11531005 and No.11401318).}
\section{Introduction}
In this paper, we focus on the eigenvalue problem $\triangle^2u=\lambda^2u$ in $\Omega$, where $\triangle$ is the Laplacian operator. We always assume that $\Omega\subseteq\mathbb{R}^n(n\geq2)$ is a bounded domain and $\partial\Omega$ is of $C^{\infty}$ piecewise analytic and $\partial\Omega$ is analytic except a set $\Gamma\subseteq\partial\Omega$ which is a finite union of some $(n-2)$ dimensional submanifolds of $\partial\Omega$. Through this paper, we always consider the following homogeneous boundary condition:
\begin{equation}\label{boundary conditions}
B_ju=0,\quad j=1,2,\quad on \quad \partial\Omega,
\end{equation}
where $B_ju=\sum\limits_{|\alpha|\leq3}a_{\alpha,j}(x)D^{\alpha}u(x)$ on $\partial\Omega$, and we also assume that the constants $a_{\alpha,j}(x)$ are all $C^{\infty}$ and piecewise analytic on $\partial\Omega$, and analytic on $\partial\Omega\setminus\Gamma$.

In 1979, F.J.Almgren in $\cite{F.J.Almgren}$ introduced the frequency concept for the harmonic functions. Then in 1986 and 1987, N.Garofalo and F.H.Lin in $\cite{N.Garofalo and F.H.Lin1}$ and $\cite{N.Garofalo and F.H.Lin2}$ established the monotonicity formula for the frequency functions and the doubling conditions for solutions of  uniformly linear elliptic equations of second order, and obtained the unique continuation property of such solutions. In 1991, F.H.Lin in $\cite{F.H.Lin}$ investigated the measure estimates of nodal sets of solutions to some uniformly linear elliptic equations of second order. In 2000, Q.Han, R.Hardt and F.H.Lin in $\cite{Q.Han R.Hardt and F.H.Lin}$ showed the structure of the nodal sets of solutions for a class of uniformly linear elliptic equations of higher order. It is pointed out that the nodal sets of solutions to higher order elliptic equations are very different from those of solutions to second order elliptic equations. For uniformly linear elliptic operators of second order, Hausdorff measures of critical zero sets of solutions are at most $n-2$ dimensional. But for uniformly linear elliptic operators of higher order, this conclusion is not always true. In fact, it is showed in $\cite{Q.Han R.Hardt and F.H.Lin}$ that, if $u$ is a solution of an $l-$th order homogeneous uniformly elliptic equation, then the set $\left\{x\in\mathbb{R}^n|u(x)=0,|\nabla u(x)|=0,\cdots, |\nabla^{l-2}u(x)|=0\right\}$ may be $n-1$ dimensional, and the set $\left\{x\in\mathbb{R}^n|u(x)=0,|\nabla u(x)|=0,\cdots, |\nabla^{l-1}u(x)|=0\right\}$ must be at most $n-2$ dimensional. The latter set is called the singular set of those solutions.  In 2003, Q.Han, R.Hardt and F.H.Lin in $\cite{Q.Han R.Hardt and F.H.Lin}$ investigated the structures and measure estimates of singular sets of solutions to uniformly linear elliptic equations of higher order. In 2014, the authors in $\cite{L.Tian and X.P.Yang}$ gave the measure estimates of nodal sets for bi-harmonic functions.

On the other hand, there have been a number of very interesting and intensive results on nodal sets for eigenfunctions to elliptic operators. S.T.Yau conjectured that, for a Laplacian eigenfunction $u$, i.e., $u$ satisfies the equation $$\triangle u+\lambda u=0$$ on a compact $n$ dimensional Riemannian manifold without boundary, it holds that
\begin{equation}\label{1.4}
c\lambda^{1/2}\leq\mathcal{H}^{n-1}\left(\left\{u(x)=0\right\}\right)\leq C\lambda^{1/2},
\end{equation}
where $\mathcal{H}^{n-1}\left(\left\{u(x)=0\right\}\right)$ denotes the $(n-1)$ dimensional Hausdorff measure of the nodal set $\left\{u=0\right\}$ on the whole manifold, $C$ is an absolute positive constant independent of $\lambda$ and $u$. This conjecture was proved for real analytic manifolds by Bruning and Yau independently in two dimensional case, and by H.Donnelly and C.Fefferman in $\cite{H.Donnelly and C.Fefferman1}$ in higher dimensional case. For $C^{\infty}$ manifolds, the following upper bound
\begin{equation}\label{1.5}
\mathcal{H}^{n-1}\left(\left\{u=0\right\}\right)\leq C\lambda^{3/4}
\end{equation}
was proved by H.Donnelly and C.Fefferman in $\cite{H.Donnelly and C.Fefferman2}$ in two dimensional case and R.T.Dong gave a very different proof for the same bound in $\cite{R.T.Dong}$ by introducing a second order frequency function. The upper bound was refined to $C\lambda^{3/4-\epsilon}$ by A.Logunov and E.Melinnikova in $\cite{A.Logunov and E.malinnikova}$. For the high dimensional case, the estimate
\begin{equation}\label{1.6}
\mathcal{H}^{n-1}\left(\left\{u=0\right\}\right)\leq C\lambda^{C\sqrt{\lambda}}
\end{equation}
was given by R.Hardt and L.Simon in $\cite{R.Hardt and L.Simon}$. In 2016, A.Logunov in $\cite{A.Logunov}$ improved the upper bound to $C\lambda^{\alpha}$ for some $\alpha>1/2$.

There are also a large number of papers concerning the lower bounds for the measure of nodal sets of Laplacian eigenfunctions(see for example $\cite{T.H.Colding and W.P.Minicozzi}$, $\cite{H.Donnelly and C.Fefferman1}$, $
\cite{C.D.Sogge and S.Zelditch}$.)

For eigenfunctions of higher order equations, the related results are very limited.  I.Kukavica in $\cite{I.Kukavica}$ gave the optimal upper measure bound of nodal sets for eigenfunctions of higher order elliptic operators in the analytic case. The eigenfunctions can be extended pass through the boundary of $\Omega$ because the operator and the boundary $\partial\Omega$ are assumed to be analytic in that paper. And then the growth order of an eigenfunction can be estimated and the measure upper bound of its nodal set can be derived. In this paper, we do not assume that the boundary $\partial\Omega$ is totally analytic. In this case, the method in $\cite{I.Kukavica}$ cannot be used directly.
So first we will define a frequency function related to an eigenfunction and establish the monotonicity formula, doubling conditions, and various estimates. As a result, we will further derive a measure upper bound of the nodal set of an eigenfunction, where the bound contains the  corresponding eigenvalue $\lambda$ and the frequency function. Next we will derive an upper bound for the frequency function, whose center is away from the nonanalytic part $\Gamma\subseteq\partial\Omega$, with the help of the doubling index of the bi-harmonic eigenfunction in a ball. We also construct an iteration scheme to control the frequency function and
extend the eigenfunction $u$ to outside of $\Omega$ by passing through the part of the boundary $\partial\Omega$ away from $\Gamma$. These lead to prove the upper bound for the measure of the nodal set of the eigenfunction $u$ in the domain $\Omega\setminus T_{R_0}(\Gamma)$ for some suitable domain
$T_{R_0}(\Gamma)=\left\{x\in\overline{\Omega}:dist(x,\Gamma)<R_0\right\}$, where $R_0$ is some suitable positive constant depending only on $n$ and $\Omega$. Finally, we locally establish an upper bound of the doubling index near the points in the nonanalytic parts by applying the upper bound of the frequency function in $\Omega$. And then by using an iteration method, we can get the desired estimation of the nodal set of $u$ near $\Gamma$.

More precisely, in order to get the upper bound for the frequency function, we first start to derive an upper bound for the frequency function centered at some point $y_0=(\bar{x},0)$ with $\bar{x}\in\Omega$ satisfying that the function value $|u(\bar{x})|$ is suitable large. Then from an iteration argument, the relationship between the frequency function and the doubling index, the ``changing center'' property for the frequency function, and the doubling conditions, we can get the upper bound for the frequency function centered at some point away from $\Gamma$. Then by the similar iteration argument again, we can get the upper bound for the frequency function at any point away from $\Gamma$.
For the frequency centered near $\Gamma$, we start from some point away from $\Gamma$, whose doubling index is controlled by $\sqrt{\lambda}$. Then by using the iteration arguments, the analytically extending of $u$, the relationship between the frequency function and the doubling index, and the doubling conditions, we can get the upper bound for the frequency function centered near $\Gamma$.
Actually one notes that the bound for the frequency function may go to infinity when the corresponding center tends to the nonanalytic part $\Gamma$. However, because the dimension of $\Gamma$ is $(n-2)$, and the upper bound for the frequency function goes to infinity not very fast as the center of the frequency function tends to $\Gamma$, the measure for the nodal set of $u$ in a domain near $\Gamma$ can be obtained by using the iteration method.

The main result of this paper is as follows.
\begin{theorem}
Let $u$ be a solution to the following boundary value problem:
\begin{equation}
\begin{cases}
\triangle^2u=\lambda^2u\quad in\quad\Omega,\\
B_ju=0,\quad j=1,2,\quad on\quad\partial\Omega,
\end{cases}
\end{equation}
where $B_j$, $j=1,2$ are linear boundary operators defined as $B_ju=\sum\limits_{|\alpha|\leq3}a_{\alpha,j}(x)D^{\alpha}u(x)$ on $\partial\Omega$, where $a_{\alpha,j}(x)$ are all $C^{\infty}$ and piecewise analytic on $\partial\Omega$, and analytic on $\partial\Omega\setminus\Gamma$.  Then the measure of nodal set of $u$ has the following estimate:
\begin{equation}
\mathcal{H}^{n-1}\left(\left\{x\in\Omega|u(x)=0\right\}\right)\leq C\sqrt{\lambda},
\end{equation}
where $C$ is a positive constant depending only on $n$, $\Omega$ and the boundary operators $B_j$, $j=1,2$.
\end{theorem}

The rest of this paper is organized as follows. In the second section, we give the definition of a frequency function related to bi-harmonic eigenfunctions and show some interesting properties including the monotonicity formula and some estimates for the frequency function. In the third section, we obtain some doubling conditions and show the ``changing center'' property for the monotonicity formula. In the fourth section, we give a measure estimate of nodal sets of eigenfunctions in $\Omega(R_0,\Gamma)$. Here $\Omega(R_0,\Gamma)\subseteq\overline{\Omega}$
 and the distance of $\Omega(R_0,\Gamma)$ and $\Gamma$ is greater than some constant $R_0$
 depending only on $n$ and $\Omega$.  Finally in the last section, we give the measure estimate of the nodal set of the eigenfunction $u$ near the nonanalytic set $\Gamma$.

\section{Frequency function}
We rewrite the equation $\triangle^2u=\lambda^2u$ as the following forms:
\begin{eqnarray*}
(\triangle-\lambda)u&=&v,\\
(\triangle+\lambda)v&=&0.
\end{eqnarray*}
Let $g(x,x_{n+1})=u(x)e^{\sqrt{\lambda}x_{n+1}}$ and $h(x,x_{n+1})=v(x)e^{\sqrt{\lambda}x_{n+1}}$. Thus we have
\begin{eqnarray*}
(\triangle-2\lambda)g&=&h,\\
\triangle h&=&0,
\end{eqnarray*}
where $g$ and $h$ are all considered as functions of $n+1$ dimensional variable $y=(x,x_{n+1})$.
Then we define the frequency function centered at $y_0$ as follows:
\begin{equation}\label{definition of frequency}
N(y_0,r)=r\frac{\int_{B_r(y_0)}\left(|\nabla g|^2+|\nabla h|^2+gh+2\lambda g^2\right)dy}{\int_{\partial B_r(y_0)}\left(g^2+h^2\right)d\sigma},
\end{equation}
where $d\sigma$ is the $n$ dimensional Hausdorff measure on $\partial B_r(y_0)$. We use the notation
$$
D_1(y_0,r)=\int_{B_r(y_0)}|\nabla g|^2dy,\quad D_2(y_0,r)=\int_{B_r(y_0)}|\nabla h|^2dy,
$$
$$
D_3(y_0,r)=\int_{B_r(y_0)}ghdy,\quad D_4(y_0,r)=2\lambda\int_{B_r(y_0)}g^2dy,
$$
$$
H_1(y_0,r)=\int_{\partial B_r(y_0)}g^2d\sigma,\quad H_2(y_0,r)=\int_{\partial B_r(y_0)}h^2d\sigma,
$$
$$
D(y_0,r)=D_1(y_0,r)+D_2(y_0,r)+D_3(y_0,r)+D_4(y_0,r),
$$
$$
H(y_0,r)=H_1(y_0,r)+H_2(y_0,r).
$$
 Then the frequency function can be written as $$N(y_0,r)=r\frac{D(y_0,r)}{H(y_0,r)}=r\frac{D_1(y_0,r)+D_2(y_0,r)+D_3(y_0,r)+D_4(y_0,r)}{H_1(y_0,r)+H_2(y_0,r)},$$
 and it is also easy to check that
 \begin{equation}\label{another form of frequency}
 N(y_0,r)=r\frac{\int_{\partial B_r(y_0)}(gg_n+hh_n)d\sigma}{\int_{\partial B_r(y_0)}(g^2+h^2)d\sigma},
 \end{equation}
where $g_n$ and $h_n$ mean $\nabla g\cdot \overrightarrow{n}$ and $\nabla h\cdot\overrightarrow{n}$ respectively, and $\overrightarrow{n}$ is the outer unit normal vector on $\partial B_r(y_0)$.

 Such a frequency function has following properties.

 \begin{lemma}\label{vanishing order property for frequency function in tirangle case}
 If the vanishing orders of $g$ and $h$ at point $y_0$ are $k$ and $l$ respectively, then it holds that
 \begin{equation}\label{lemma vanishing order property for frequency function in tirangle case}
 \lim\limits_{r\longrightarrow0+}N(y_0,r)=\min\left\{k,l\right\}.
 \end{equation}
 \end{lemma}

 \begin{proof}
 Without loss of generality, we may assume that $y_0=0$. Because the vanishing order of $g$ at the origin is $k$, we have
 \begin{equation*}
 g(y)=P_k(y)+o(|y|^k),
 \end{equation*}
 where $P_k(y)$ is a homogeneous polynomial of degree $k$. Similarly, we have
 \begin{equation*}
 h(y)=P_l(y)+o(|y|^l),
 \end{equation*}
 where $P_l(y)$ is a homogeneous polynomial of degree $l$. In the polar coordinate system $(r,\theta)=(r,\theta_1,\theta_2,\cdots,\theta_n)$, $P_k(y)$ and $P_l(y)$ can be written as the following forms.
 \begin{equation*}
 P_k(y)=r^k\phi(\theta),\quad P_l(y)=r^l\psi(\theta).
 \end{equation*}
 Then from $(\ref{another form of frequency})$, we have
 \begin{eqnarray*}
 N(0,r)&=&r\frac{\int_{\partial B_r}\left(gg_n+hh_n\right)d\sigma}{\int_{\partial B_r}\left(g^2+h^2\right)d\sigma}
 \\&=&
 \frac{\int_{\partial B_1}\left(kr^{2k+n}\phi^2(\theta)+lr^{2l+n}\psi^2(\theta)+o(r^{2k+n})+o(r^{2l+n})\right)d\sigma}{\int_{\partial B_1}\left(r^{2k+n}\phi^2(\theta)+r^{2l+n}\psi^2(\theta)+o(r^{2k+n})+o(r^{2l+n})\right)d\sigma}.
 \end{eqnarray*}
 Letting $r\longrightarrow0+$, we have the desired result.
 \end{proof}

 \begin{lemma}\label{lower bound for the frequency}
 The frequency defined in $(\ref{definition of frequency})$ has the following lower bound:
 \begin{equation}\label{lemma lower bound for the frequency}
 N(y_0,r)\geq-\frac{r^2}{8(n+1)\lambda}.
 \end{equation}
 \end{lemma}

\begin{proof}
Without loss of generality, we may assume that $y_0=0$. Then
\begin{equation*}
\int_{B_r(0)}ghdy\leq2\lambda\int_{B_r(0)}g^2dy+\frac{1}{8\lambda}\int_{B_r(0)}h^2dy.
\end{equation*}
Because $h$ is a harmonic function, we have
\begin{equation*}
\int_{B_r(0)}h^2dy\leq\frac{r}{n+1}\int_{\partial B_r(0)}h^2d\sigma,
\end{equation*}
which can be seen in $\cite{Q.Han and F.H.Lin}$.
So
\begin{eqnarray*}
N(0,r)&\geq& r\frac{\int_{B_r(0)}\left(|\nabla g|^2+|\nabla h|^2\right)dy-\frac{r}{8(n+1)\lambda}\int_{\partial B_r(0)}h^2d\sigma}{\int_{\partial B_r(0)}\left(g^2+h^2\right)d\sigma}
\\&\geq&-\frac{r^2}{8(n+1)\lambda},
\end{eqnarray*}
which is the desired result.
\end{proof}

Now we will show the monotonicity formula.

\begin{theorem}\label{monotonicity formula}
Let $\lambda>1$. Then there exist positive constants $C_0$ and $C$ depending only on $n$, such that if $N(y_0,r)\geq C_0$ then
\begin{equation}\label{lemma monotonicity formula}
\frac{d\ln N(y_0,r)}{dr}\geq-C
\end{equation}
for $r\leq1$.
\end{theorem}

\begin{proof}
Let $y_0=0$. Note that
\begin{equation*}
\frac{d\ln N(0,r)}{dr}=\frac{1}{r}+\sum\limits_{i=1}^4\frac{D_i'(0,r)}{D(0,r)}-\sum\limits_{j=1}^2\frac{H_j'(0,r)}{H(0,r)}.
\end{equation*}

First we calculate the term $D_1'(0,r)$.
\begin{eqnarray*}
D_1'(0,r)&=&\int_{\partial B_r(0)}|\nabla g|^2d\sigma
\\&=&\int_{\partial B_r(0)}|\nabla g|^2\cdot\frac{y}{r}\cdot\frac{y}{r}d\sigma
\\&=&\frac{1}{r}\int_{B_r(0)}div\left(|\nabla g|^2y\right)dy
\\&=&\frac{n}{r}\int_{B_r(0)}|\nabla g|^2dy+\frac{2}{r}\int_{B_r(0)}\nabla g\cdot\nabla^2g\cdot ydy
\\&=&\frac{n-2}{r}D_1(0,r)+2\int_{\partial B_r(0)}g_n^2d\sigma-\frac{2}{r}\int_{B_r}\left(h+2\lambda g\right)\nabla g\cdot ydy.
\end{eqnarray*}
Here $d\sigma$ means the $n-1$ dimensional Hausdorff measure on $\partial B_r(0)$.
Similarly, we have
\begin{equation*}
D_2'(0,r)=\frac{n-2}{r}D_2(0,r)+2\int_{\partial B_r(0)}h_n^2d\sigma.
\end{equation*}
For $D_3'(0,r)$ and $D_4'(0,r)$, we have
\begin{eqnarray*}
D_3'(0,r)&=&\int_{\partial B_r(0)}ghd\sigma=\frac{1}{r}\int_{B_r(0)}div(ghy)dy
\\&=&\frac{n-2}{r}D_3(0,r)+\frac{1}{r}\left(\int_{B_r(0)}h\nabla g\cdot ydy+\int_{B_r(0)}g\nabla h\cdot ydy+2\int_{B_r(0)}ghdy\right),
\end{eqnarray*}
and
\begin{eqnarray*}
D_4'(0,r)&=&2\lambda\int_{\partial B_r(0)}g^2d\sigma
\\&=&\frac{n-2}{r}D_4(0,r)+\frac{2\lambda}{r}\left(\int_{B_r(0)}2g\nabla g\cdot ydy+2\int_{B_r(0)}g^2dy\right).
\end{eqnarray*}
So
\begin{eqnarray*}
D'(0,r)&=&\frac{n-2}{r}D(0,r)+2\int_{\partial B_r(0)}\left(g_n^2+h_n^2\right)d\sigma
\\&-&\frac{1}{r}\int_{B_r(0)}h\nabla g\cdot ydy+\frac{1}{r}\int_{B_r(0)}g\nabla h\cdot ydy
\\&+&\frac{2}{r}\int_{B_r(0)}ghdy+\frac{4\lambda}{r}\int_{B_r(0)}g^2dy
\\&\geq&\frac{n-2}{r}D(0,r)+2\int_{\partial B_r(0)}\left(g_n^2+h_n^2\right)d\sigma
\\&-&\int_{B_r(0)}|h||\nabla g|dy-\int_{B_r(0)}|g||\nabla h|dy-\frac{1}{4r\lambda}\int_{B_r(0)}h^2dy.
\end{eqnarray*}
For the term $\int_{B_r(0)}|h||\nabla g|dy$ and $\int_{B_r(0)}|g||\nabla h|dy$, we have
\begin{equation*}
\int_{B_r(0)}|h||\nabla g|dy\leq\frac{r}{2}\int_{B_r(0)}|\nabla g|^2dy+\frac{1}{2r}\int_{B_r(0)}h^2dy,
\end{equation*}
and
\begin{equation*}
\int_{B_r(0)}|g||\nabla h|dy\leq\frac{r}{2}\int_{B_r(0)}|\nabla h|^2dy+\frac{1}{2r}\int_{B_r(0)}g^2dy.
\end{equation*}
Now we consider the term $\int_{B_r(0)}g^2dy$.
Note that $\triangle g=h+2\lambda g$, we can separate $g=\overline{g}+\underline{g}$, where $\overline{g}$ and $\underline{g}$ satisfy that
\begin{equation}
\begin{cases}
\triangle\overline{g}=h+2\lambda g\quad in\quad B_r(0),\\\nonumber
\overline{g}=0\quad on\quad \partial B_r(0),\nonumber
\end{cases}
\end{equation}
and
\begin{equation*}
\begin{cases}
\triangle\underline{g}=0\quad in\quad B_r(0),\\
\underline{g}=g\quad on\quad \partial B_r(0).
\end{cases}
\end{equation*}
Write the term $\int_{B_r(0)}g^2dy$ as follows:
\begin{equation*}
\int_{B_r(0)}g^2dy\leq2\int_{B_r(0)}\underline{g}^2dy+2\int_{B_r(0)}\overline{g}^2dy.
\end{equation*}
Now we consider $\overline{g}$ and $\underline{g}$.
Because
\begin{equation*}
\int_{B_r(0)}\triangle\overline{g}\cdot\overline{g}dy=\int_{B_r(0)}h\overline{g}dy+2\lambda\int_{B_r(0)}\overline{g}^2dy
+2\lambda\int_{B_r(0)}\overline{g}\underline{g}dy,
\end{equation*}
it holds that
\begin{eqnarray*}
2\lambda\int_{B_r(0)}\overline{g}^2dy&=&-2\lambda\int_{B_r(0)}\overline{g}\underline{g}dy-\int_{B_r(0)}h\overline{g}dy
+\int_{B_r(0)}\triangle\overline{g}\cdot\overline{g}dy
\\&\leq&4\lambda\int_{B_r(0)}\underline{g}^2dy+\lambda\int_{B_r(0)}\overline{g}^2dy+\frac{1}{2\lambda}\int_{B_r(0)}h^2dy
\\&+&\frac{\lambda}{2}\int_{B_r(0)}\overline{g}^2dy-\int_{B_r(0)}|\nabla\overline{g}|^2dy
\\&\leq&\frac{3\lambda}{2}\int_{B_r(0)}\overline{g}^2dy+4\lambda\int_{B_r(0)}\underline{g}^2dy
+\frac{1}{2\lambda}\int_{B_r(0)}h^2dy.
\end{eqnarray*}
So
\begin{equation*}
\int_{B_r(0)}\overline{g}^2dy\leq8\int_{B_r(0)}\underline{g}^2dy+\frac{1}{\lambda^2}\int_{B_r(0)}h^2dy.
\end{equation*}
Because $\underline{g}$ and $h$ both are harmonic functions, we have
\begin{equation*}
\int_{B_r(0)}\underline{g}^2dy\leq\frac{r}{n+1}\int_{\partial B_r(0)}g^2d\sigma,
\end{equation*}
and
\begin{equation*}
\int_{B_r(0)}h^2dy\leq\frac{r}{n+1}\int_{\partial B_r(0)}h^2d\sigma.
\end{equation*}
By the above arguments, we have that
\begin{eqnarray*}
D'(0,r)&\geq&\frac{n-2}{r}D(0,r)+2\int_{\partial B_r(0)}\left(g_n^2+h_n^2\right)d\sigma
\\&-&\frac{r}{2}\int_{B_r(0)}\left(|\nabla g|^2+|\nabla h|^2\right)dy-CH(0,r),
\end{eqnarray*}
where $C$ is a positive constant depending only on $n$.

On the other hand, from the assumption that $N(0,r)\geq C_0$, where $C_0$ is a constant to be determined, we have
\begin{equation*}
H(0,r)\leq r\frac{D(0,r)}{C_0}.
\end{equation*}
So
\begin{equation*}
D'(0,r)\geq\frac{n-2}{r}D(0,r)+2\int_{\partial B_r(0)}\left(g_n^2+h_n^2\right)dy-\frac{Cr}{C_0}D(0,r)-\frac{1}{2}\int_{B_r(0)}\left(|\nabla g|^2+|\nabla h|^2\right)dy.
\end{equation*}
Also note that
\begin{eqnarray*}
D(0,r)&=&\int_{B_r(0)}\left(|\nabla g|^2+|\nabla h|^2+hg+2\lambda g^2\right)dy
\\&\geq&\int_{B_r(0)}\left(|\nabla g|^2+|\nabla h|^2\right)dy-\frac{1}{8\lambda}\int_{B_r(0)}h^2dy
\\&\geq&\int_{B_r(0)}\left(|\nabla g|^2+|\nabla h|^2\right)dy-\frac{r}{8(n+1)\lambda}\int_{\partial B_r(0)}h^2dy
\\&\geq&\int_{B_r(0)}\left(|\nabla g|^2+|\nabla h|^2\right)dy-\frac{r}{8(n+1)\lambda}H(0,r)
\\&\geq&\int_{B_r(0)}\left(|\nabla g|^2+|\nabla h|^2\right)dy-\frac{1}{8(n+1)C_0\lambda}D(0,r).
\end{eqnarray*}
By choosing $C_0=1/(n+1)$, we have
\begin{equation*}
D(0,r)\geq\frac{8}{9}\int_{B_r(0)}\left(|\nabla g|^2+|\nabla h|^2\right)dy.
\end{equation*}
So
\begin{equation*}
D'(0,r)\geq\frac{n-2}{r}D(0,r)+2\int_{\partial B_r(0)}\left(g_n^2+h_n^2\right)d\sigma-CD(0,r).
\end{equation*}
By some direct calculation, we have
\begin{equation*}
H'(0,r)=\frac{n-2}{r}H(0,r)+2\int_{\partial B_r(0)}\left(hh_n+gg_n\right)d\sigma.
\end{equation*}
So
\begin{equation*}
\frac{d\ln N(0,r)}{dr}\geq-Cr,
\end{equation*}
which is just the result we want.
\end{proof}

From this ``monotonicity formula'', the following corollary can be got easily.

\begin{corollary}\label{frequency control}
Let $\lambda>1$. Then for any $0<r_1<r_2\leq 1$, we have
\begin{equation}\label{frequency control_big radius control small radius}
N(0,r_1)\leq C\max\left\{C_0, N(0,r_2)\right\}.
\end{equation}
Moreover, if for any $r\in(r_1,r_2)$, it holds that $N(0,r)>C_0$, then
\begin{equation}\label{frequency control_small radius control big radius}
N(0,r_2)\geq C'N(0,r_1).
\end{equation}
Here $C_0$ is the same constant as in Theorem $\ref{monotonicity formula}$; $C$ and $C'$ are positive constants depending only on $n$.
\end{corollary}

\begin{proof}
Let $\mathcal{I}=\left\{r\in(0,r_2)|N(0,r)>C_0\right\}$. Then $\mathcal{I}$ is a union of at most countable number of intervals $(a_i,b_i)$, with $N(0,a_i)=C_0$, $N(0,b_i)=C_0$ if $b_i\neq r_2$.  If $r_1\in (a_i,b_i)$ with $b_i\neq r_2$, then from Theorem $\ref{monotonicity formula}$, it holds that
\begin{eqnarray*}
\ln\frac{N(0,b_i)}{N(0,r_1)}=\int_{r_1}^{b_i}d\ln N(0,\rho)\\&\geq&-C(r_2-r_1).
\end{eqnarray*}
So
\begin{equation*}
N(0,r_1)\leq C_0e^{C(r_2-r_1)}\leq CC_0.
\end{equation*}
If $r_1\in(a_i,b_i)$ with $B_i=r_2$, then by the same arguments, we have
\begin{equation}\label{integrate monotonicity formula}
\ln\frac{N(0,r_2)}{N(0,r_1)}=\int_{r_1}^{r_2}d\ln N(0,\rho)\geq-C(r_2-r_1),
\end{equation}
and thus
\begin{equation*}
N(0,r_1)\leq CN(0,r_2).
\end{equation*}
Then we get the first result, i.e., $(\ref{frequency control_big radius control small radius})$.

If for any $r\in(r_1,r_2)$, it holds that $N(0,r)>C_0$, then $r_1$ and $r_2$ must be in the same interval $(a_i,b_i)$. So from $(\ref{frequency control_big radius control small radius})$, the second result can be got directly.
\end{proof}

\section{Doubling conditions}

In this section, we show the doubling conditions based on the monotonicity formula. First we can  show the following doubling condition including both $g$ and $h$.

\begin{lemma}\label{doubling condition for both u and v}
Let $\lambda>1$. Then for any $0<r_1<r_2<1$, it holds that
\begin{equation}\label{first doubling condition}
\fint_{\partial B_{r_2}(y_0)}(g^2+h^2)d\sigma
\leq\left(\frac{r_2}{r_1}\right)^{C\max\left\{N(y_0,r_2), C_0\right\}}
\fint_{\partial B_{r_1}(y_0)}(g^2+h^2)d\sigma,
\end{equation}

\begin{equation}\label{second doubling condition}
\fint_{B_{r_2}(y_0)}(g^2+h^2)dy
\leq\left(\frac{r_2}{r_1}\right)^{C'\max\left\{N(y_0,r_2), C_0\right\}}
\fint_{B_{r_1}(y_0)}(g^2+h^2)dy,
\end{equation}
and
\begin{equation}\label{third doubling condition}
\fint_{\partial B_{r_2}(y_0)}(g^2+h^2)dy
\geq\left(\frac{r_2}{r_1}\right)^{\bar{C}\min\limits_{[r_1,r_2]}
N(y_0,r)}\fint_{\partial B_{r_1}(y_0)}(g^2+h^2)dy,
\end{equation}

\begin{equation}\label{fourth doubling condition}
\fint_{B_{r_2}(y_0)}(g^2+h^2)dy\geq
\left(\frac{r_2}{r_1}\right)^{\bar{C}'\min\limits_{[r_1,r_2]}N(y_0,r)}
\fint_{B_{r_1}(y_0)}(g^2+h^2)dy,
\end{equation}
where $C_0$ is the same constant as in Theorem $\ref{monotonicity formula}$; $C$, $C'$, $\bar{C}$ and $\bar{C}'$ are all positive constants depending only on $n$.
\end{lemma}

\begin{proof}
We only need to prove $(\ref{first doubling condition})$ and $(\ref{third doubling condition})$. The inequality $(\ref{second doubling condition})$ and $(\ref{fourth doubling condition})$ can be easily deduced from $(\ref{first doubling condition})$ and $(\ref{third doubling condition})$.
Without loss of generality, we assume that $y_0$ is the origin. Define
\begin{equation}\label{H(r)}
\bar{H}(\rho)=\fint_{\partial B_{\rho}(0)}(g^2+h^2)d\sigma=\frac{1}{\omega_n\rho^n}\int_{\partial B_{\rho}}(g^2+h^2)d\sigma,
\end{equation}
where $\omega_n$ is the Hausdorff measure of an $n$ dimensional unit sphere.
It is easy to check that
\begin{equation}\label{dlnH(r)/dr}
\frac{d\ln\bar{H}(0,\rho)}{d\rho}=\frac{2N(0,\rho)}{\rho}.
\end{equation}
So from Theorem $\ref{monotonicity formula}$ and Corollary $\ref{frequency control}$, it holds that
\begin{eqnarray*}
\int_{r_1}^{r_2}\frac{2N(0,\rho)}{\rho}d\rho&\leq&\max\limits_{\rho\in[r_1,r_2]}\left\{N(0,\rho)\right\}\ln\frac{r_2}{r_1}
\\&\leq&C\max\left\{C_0,N(0,r_2)\right\}\ln\frac{r_2}{r_1}.
\end{eqnarray*}
So we have
\begin{equation*}
\bar{H}(r_2)\leq\left(\frac{r_2}{r_1}\right)^{C\max\left\{C_0,N(0,r_2)\right\}}\bar{H}(r_1).
\end{equation*}

On the other hand, it also holds that
\begin{equation*}
\int_{r_1}^{r_2}\frac{2N(0,\rho)}{\rho}d\rho
\geq\min\limits_{\rho\in[r_1,r_2]}\left\{N(0,\rho)\right\}\ln\frac{r_2}{r_1}.
\end{equation*}
So we also have
\begin{equation*}
\bar{H}(r_2)\geq\left(\frac{r_2}{r_1}\right)^{\min\limits_{\rho\in[r/2,r]}N(0,\rho)}\bar{H}(r_1).
\end{equation*}
Thus we get the desired result.
\end{proof}

Basing on this doubling condition, we can prove the following doubling condition only for $g$.

\begin{lemma}\label{doubling condition only for u or v}
Let $\lambda>1$ and $0<r<1$. Then
\begin{equation}\label{doubling condition only for u}
\fint_{B_r(y_0)}g^2dy\leq\left(\frac{1}{r^4}+\lambda^2\right)
2^{C\max\left\{C_0,N(y_0,r)\right\}}\fint_{B_{\frac{r}{2}}(y_0)}g^2dy,
\end{equation}
where $C$ is a positive constant depending only on $n$.
\end{lemma}

In order to prove this doubling condition, we need the following interior estimation.
\begin{lemma}\label{h^2 controled by g^2}
Let $\lambda>1$. Then for $0<r<1/2$, we have
\begin{equation}\label{formula h^2 controled by g^2}
\int_{B_{\frac{r}{2}}(y_0)}h^2dy\leq C\left(\frac{1}{r^4}+\lambda^2\right)\int_{B_r(y_0)}g^2dy,
\end{equation}
where $C$ is a positive constant depending only on $n$.
\end{lemma}

\begin{proof}
Without loss of generality, let $y_0=0$. Let $\phi$ be the cut-off function satisfying that
$$
\phi=1\quad in \quad B_{\frac{r}{2}};\quad
\phi=0\quad outside \quad B_{r};$$$$
0\leq\phi\leq1;\quad
|\nabla\phi|\leq c/r;\quad $$
and
$$\quad
|\nabla^2\phi|\leq c/r^2,
$$
where $c>1$ is a positive constant depending only on $n$. Define a test function $\psi$ as follows:
\begin{eqnarray}\label{definition of test function}
\psi &=&e^{1-\frac{1}{\phi}}\in(0,1)\quad if\quad \phi>0,\\\nonumber
\psi &=&0 \quad if\quad\phi=0.
\end{eqnarray}
Then, up to a limit, it also holds that
\begin{equation}\label{property of test function}
\frac{\psi}{\phi^k}=0\quad if \quad \phi=0
\end{equation}
for any positive constant $k$.
From the equation $\triangle g=h+2\lambda g$, we have
\begin{equation}\label{equation times test function}
\int_{B_r(0)}\triangle gh\psi dy=\int_{B_r(0)}h^2\psi dy+2\lambda\int_{B_r(0)}gh\psi dy.
\end{equation}
First calculate the left hand side.
\begin{eqnarray*}
\int_{B_r(0)}\triangle gh\psi dy&=&\int_{B_r(0)}g\triangle(h\psi)dy
\\&=&\int_{B_r(0)}g\triangle h\psi dy+2\int_{B_r(0)}g\nabla h\nabla \psi dy
+\int_{B_r(0)}gh\triangle\psi dy.
\end{eqnarray*}
Because
\begin{eqnarray*}
\triangle h&=&0,\\
\nabla \psi&=&\psi\cdot\frac{\nabla\phi}{\phi^2},\\
\triangle \psi&=&\psi\left(\frac{|\nabla\phi|^2}{\phi^4}
-\frac{2|\nabla\phi|^2}{\phi^3}+\frac{\triangle\phi}{\phi^2}\right),
\end{eqnarray*}
we have
\begin{equation*}
\int_{B_r(0)}\triangle gh\psi dy=
2\int_{B_r(0)}g\psi\frac{1}{\phi^2}\nabla h\cdot\nabla\phi dy
+\int_{B_r(0)}gh\psi\left(\frac{|\nabla\phi|^2}{\phi^4}-\frac{2|\nabla\phi|^2}{\phi^3}
+\frac{\triangle\phi}{\phi^2}\right)dy.
\end{equation*}

From $(\ref{equation times test function})$, we have
\begin{eqnarray*}
\int_{B_r(0)}h^2\psi dy&=&\int_{B_r(0)}\triangle gh\psi dy-2\lambda\int_{B_r(0)}gh\psi dy
\\&\leq&2\int_{B_r(0)}g\psi \frac{1}{\phi^2}\nabla h\nabla\phi dy
+\int_{B_r(0)}gh\psi \left(\frac{|\nabla\phi|^2}{\phi^4}-\frac{2|\nabla\phi|^2}{\phi^3}
+\frac{\triangle\phi}{\phi^2}\right)dy
\\&+&2\lambda^2\int_{B_r(0)}g^2\psi dy+\frac{1}{2}\int_{B_r(0)}h^2\psi dy.
\end{eqnarray*}
Because $\triangle h=0$, we have
\begin{equation*}
\int_{B_r(0)}\triangle hh\psi\phi^4dy=0.
\end{equation*}
Thus
\begin{equation}\label{another midel case}
\int_{B_r(0)}|\nabla h|^2\psi\phi^4dy+\int_{B_r(0)}h\nabla h\cdot\nabla\phi \psi\left(\phi^2+4\phi^3\right)dy=0.
\end{equation}
So
\begin{eqnarray}\label{midel case}
\int_{B_r(0)}|\nabla h|^2\psi\phi^4dy&=&-\int_{B_r(0)}h\nabla h\nabla\phi \psi\left(\phi^2+4\phi^3\right)dy
\\\nonumber
&\leq&
5\int_{B_r(0)}h|\nabla h||\nabla\phi| \psi\phi^2dy
\\&\leq&\frac{1}{2}\int_{B_r(0)}|\nabla h|^2\psi\phi^4dy
+18\int_{B_r(0)}h^2\psi|\nabla\phi|^2dy.
\end{eqnarray}
Thus
\begin{equation*}
\int_{B_r(0)}|\nabla h|^2\psi\phi^4dy\leq\frac{36c^2}{r^2}\int_{B_r(0)}h^2\psi dy.
\end{equation*}
So
\begin{eqnarray*}
\int_{B_r(0)}g\nabla h\nabla\phi \psi\frac{1}{\phi^2}dy
&\leq&
\frac{r^2}{144c^2}\int_{B_r(0)}|\nabla h|^2\psi\phi^4dy
+\frac{36c^2}{r^2}\int_{B_r(0)}g^2\psi\frac{1}{\phi^8}|\nabla\phi|^2dy
\\&\leq&
\frac{1}{4}\int_{B_r(0)}h^2\psi dy
+\frac{36c^4}{r^4}\int_{B_r(0)}g^2\psi\frac{1}{\phi^8}dy.
\end{eqnarray*}

On the other hand, we also have
\begin{equation*}
2\lambda\int_{B_r(0)}gh\psi\leq\frac{1}{8}\int_{B_r(0)}h^2\psi dy
+8\lambda^2\int_{B_r(0)}g^2\psi dy,
\end{equation*}
and
\begin{eqnarray*}
\int_{B_r(0)}gh\psi\left(\frac{|\nabla\phi|^2}{\phi^4}-2\frac{|\nabla\phi|^2}{\phi^3}
+\frac{\triangle\phi}{\phi^2}\right)dy
&\leq&\frac{4c^2}{r^2}\int_{B_r(0)}|g||h|\psi\frac{1}{\phi^4}dy
\\&\leq&
\frac{1}{16}\int_{B_r(0)}h^2\psi dy
+\frac{64c^4}{r^4}\int_{B_r(0)}h^2\psi\frac{1}{\phi^8}dy.
\end{eqnarray*}
Put these inequalities into $(\ref{another midel case})$, it holds that
\begin{equation*}
\frac{1}{16}\int_{B_r(0)}h^2\psi dy\leq10\lambda^2\int_{B_r(0)}g^2\psi dy
+\frac{100c^4}{r^4}\int_{B_r(0)}g^2\psi\frac{1}{\phi^8}dy.
\end{equation*}
From $(\ref{definition of test function})$ and $(\ref{property of test function})$, we know that $0\leq\psi\leq1$ when $0\leq\phi\leq1$, $\psi=1$ when $\phi=1$, $\psi=0$ when $\phi=0$, and $\psi/\phi^k=0$ when $\phi=0$ for any positive integer $k$. So we have
\begin{equation*}
\int_{B_\frac{r}{2}(0)}h^2dy\leq C\left(\lambda^2+\frac{1}{r^4}\right)\int_{B_r(0)}g^2dy,
\end{equation*}
where $C$ is a positive constant depending only on $c$ and $n$, and thus only on $n$. That is just the desired result.
\end{proof}

Now we will present the proof of Lemma $\ref{doubling condition only for u or v}$.

\textbf{Proof of Lemma $\ref{doubling condition only for u or v}$:}

Without loss of generality, we assume that $y_0$ is the origin. Then by using Lemma $\ref{h^2 controled by g^2}$ and Lemma $\ref{doubling condition for both u and v}$, we have
\begin{eqnarray*}
\int_{B_r(0)}g^2dy&\leq&\int_{B_r(0)}(g^2+h^2)dy
\\&\leq&
2^{C\max\left\{N(0,r),C_0\right\}}\int_{B_{\frac{r}{4}}(0)}(g^2+h^2)dy
\\&\leq&2^{C\max\left\{N(0,r),C_0\right\}}\left(\lambda^2+\frac{1}{r^4}\right)\int_{B_{\frac{r}{2}}(0)}g^2dy,
\end{eqnarray*}
which is the desired result. \qed

\begin{remark}\label{remark for doubling condition only for u with any radius}
Similarly, it is easy to check that the following doubling condition also holds:
\begin{equation}\label{doubling condition only for u with any radius}
\int_{B_{r_2}(0)}g^2dy\leq \left(\frac{r_2}{r_1}\right)^{C\max(N(0,r_2),C_0)}
\left(\lambda^2+\frac{1}{(r_2-r_1)^4}\right)\int_{B_{r_1}(0)}g^2dy,
\end{equation}
where $C$ is a positive constant depending only on $n$ and $0<r_1<r_2<1$.
In fact, in the proof of Lemma $\ref{h^2 controled by g^2}$,
if we change $r$ and $r/2$ into $r_2$ and $r_1$, we can get that
\begin{equation*}
\int_{B_{r_1}(0)}h^2dy\leq C\left(\lambda^2+\frac{1}{(r_2-r_1)^4}\right)\int_{B_{r_2}(0)}g^2dy,
\end{equation*}
where $C$ is a positive constant depending only on $n$. Then by using the same arguments in the proof of Lemma $\ref{doubling condition only for u or v}$, we can get the desired result.
\end{remark}

At the end of this section, we will give a ``changing center'' property for the frequency function, which can be proved by using the above doubling conditions.

\begin{theorem}\label{frequency change center}
Let $\lambda>1$, and $0<r_0<1$.
Then for any $p\in B_{r_0/4}(0)$, we have
\begin{equation}\label{changing center property}
N(p,\rho)\leq C\max\left\{N(0,r_0),C_0\right\},
\end{equation}
if $\rho\leq\frac{1}{2}(r_0-|p|)$, where $C$ is a positive constant depending only on $n$.
\end{theorem}

\begin{proof}

If for some $r\geq\frac{1}{2}(r_0-|p|)$, it holds that $N(p,r)\leq C_0$, then from Corollary $\ref{frequency control}$, we can get that
\begin{equation*}
N(p,\rho)\leq CN(p,r)\leq CC_0,
\end{equation*}
which is the desired result. So in the following of the proof, we always assume that $N(0,r)>C_0$ for any $r\geq\frac{1}{2}(r_0-|p|)$.

For any $p\in B_{\frac{1}{4}r_0}$, it holds that
\begin{equation}\label{containning relationship of balls}
\bar{B}_{\frac{3}{4}r_0}(p)\subseteq\bar{B}_{r_0},\quad \bar{B}_{\frac{1}{8}r_0}\subseteq \bar{B}_{\frac{1}{2}r_0}(p).
\end{equation}
Thus from Lemma $\ref{doubling condition for both u and v}$, we have
\begin{eqnarray}\label{change ball inequality}
\fint_{B_{\frac{3}{4}r_0}(p)}(g^2+h^2)dy&\leq&\left(\frac{4}{3}\right)^{n+1}\fint_{B_{r_0}}(g^2+h^2)dy \nonumber
\\&\leq&4^{C\max\left\{C_0, N(0,r_0)\right\}}\fint_{B_{\frac{1}{8}r_0}}(g^2+h^2)dy \nonumber
\\&\leq&4^{C\max\left\{C_0, N(0,r_0)\right\}}\fint_{B_{\frac{1}{2}r_0}(p)}(g^2+h^2)dx.
\end{eqnarray}
Now we will show that
\begin{equation}\label{key claim for this lemma}
\fint_{\partial B_{\frac{5}{8}r_0}(p)}(g^2+h^2)dy\leq4^{C\max\left\{C_0,N(0,r_0)\right\}}\fint_{\partial B_{\frac{1}{2}r_0}(p)}(g^2+h^2)dx.
\end{equation}
In fact,
\begin{eqnarray*}
\int_{B_{\frac{3}{4}r_0}(p)}(g^2+h^2)dy&\geq&\int_{B_{\frac{3}{4}r_0}(p)-B_{\frac{5}{8}r_0}(p)}(g^2+h^2)dy
\\&=&\int_{\frac{5}{8}r_0}^{\frac{3}{4}r_0}\omega_nr^n\fint_{\partial B_r(p)}(g^2+h^2)d\sigma dr.
\end{eqnarray*}
Because
\begin{eqnarray*}
\ln\frac{\fint_{\partial B_r(p)}(g^2+h^2)d\sigma}{\fint_{\partial B_{\frac{5}{8}r_0}(p)}(g^2+h^2)d\sigma}
&=&2\int_{\frac{5}{8}r_0}^r\frac{N(p,\rho)}{\rho}d\rho
\\&\geq&2C_0\ln\frac{r}{\frac{5}{8}r_0}.
\end{eqnarray*}
So
\begin{eqnarray*}
\int_{B_{\frac{3}{4}r_0}(p)}(g^2+h^2)dy&\geq&\int_{\frac{5}{8}r_0}^{\frac{3}{4}r_0}\omega r^n\frac{8r}{5r_0}e^{2C_0}dr\fint_{\partial B_{\frac{5}{8}r_0}(p)}(g^2+h^2)d\sigma
\\&=&Cr_0^{n+1}\fint_{\partial B_{\frac{5}{8}r_0}(p)}(g^2+h^2)d\sigma,
\end{eqnarray*}
which means that
\begin{equation}\label{inequality 1 for key inequality}
\fint_{B_{\frac{3}{4}r_0}(p)}(g^2+h^2)dy\geq C\fint_{\partial B_{\frac{5}{8}r_0}(p)}(g^2+h^2)d\sigma.
\end{equation}

On the other hand,
\begin{eqnarray*}
\int_{B_{\frac{1}{2}r_0}(p)}(g^2+h^2)dy=\int_0^{\frac{1}{2}r_0}\omega_nr^n\fint_{\partial B_r(p)}(g^2+h^2)d\sigma dr.
\end{eqnarray*}
Because
\begin{eqnarray*}
\ln\frac{\fint_{\partial B_{\frac{1}{2}r_0}(p)}(g^2+h^2)d\sigma}{\fint_{\partial B_r(p)}(g^2+h^2)d\sigma}
&=&\int_r^{\frac{1}{2}r_0}\frac{2N(p,\rho)}{\rho}d\rho
\\&\geq&-C\frac{(\frac{1}{2}r_0-r)^3}{\lambda}.
\end{eqnarray*}
Thus
\begin{eqnarray*}
\fint_{\partial B_r(p)}(g^2+h^2)d\sigma&\leq&2^{\frac{C}{\lambda}(\frac{1}{2}r_0-r)^3}\fint_{\partial B_{\frac{1}{2}r_0}(p)}(g^2+h^2)d\sigma
\\&\leq&2^{C\frac{C}{\lambda}r_0^3}\fint_{\partial B_{\frac{1}{2}r_0}(p)}(g^2+h^2)d\sigma.
\end{eqnarray*}
So
\begin{equation}\label{inequality 2 for key inequality}
\fint_{B_{\frac{1}{2}r_0}(p)}(g^2+h^2)dy\leq C\fint_{\partial B_{\frac{1}{2}r_0}(p)}(g^2+h^2)d\sigma.
\end{equation}
It also holds that
\begin{eqnarray}\label{inequality 3 for key inequality}
\int_{B_{\frac{3}{4}r_0}(p)}(g^2+h^2)dy&\leq&\int_{B_r(0)}(g^2+h^2)dy
\\\nonumber
&\leq&2^{C\max\left\{C_0,N(0,r)\right\}}\int_{B_{\frac{1}{8}r_0}(0)}(g^2+h^2)dy
\\\nonumber
&\leq&2^{C\max\left\{C_0,N(0,r)\right\}}\int_{B_{\frac{1}{2}r_0}(p)}(g^2+h^2)dy.
\end{eqnarray}
From $(\ref{inequality 1 for key inequality})$, $(\ref{inequality 2 for key inequality})$ and $(\ref{inequality 3 for key inequality})$, we can get the inequality $(\ref{key claim for this lemma})$ directly.

Note that
\begin{equation*}
\frac{d}{dr}\ln\fint_{\partial B_r(p)}(g^2+h^2)d\sigma=\frac{2N(p,r)}{r},
\end{equation*}
we have
\begin{eqnarray*}
\ln\frac{\fint_{B_{\frac{5}{8}r_0}(p)}(g^2+h^2)d\sigma}{\fint_{B_{\frac{1}{2}r_0}(p)}(g^2+h^2)d\sigma}
&=&\int_{\frac{1}{2}r_0}^{\frac{5}{8}r_0}\frac{N(p,\rho)}{\rho}d\rho
\\&\geq&CN(p,\frac{1}{2}r_0).
\end{eqnarray*}
Thus from $(\ref{key claim for this lemma})$, we have
\begin{equation*}
N(p,\frac{1}{2}r_0)\leq C\max\left\{C_0,N(0,\frac{5}{8}r_0)\right\},
\end{equation*}
which is the desired result.
\end{proof}

\section{Measure estimates of the nodal sets away from non-analytic parts}

In this sectioin, we will give a measure estimate of the nodal set for an eigenfunction $u$ in a subset of $\Omega$ away from $\Gamma$, the non-analytic part of $\partial\Omega$. We divide our task into two steps. First we will show the measure upper bounds for the nodal set of $u$ in some small balls in terms of the frequency function. Then we will give the upper bound for the frequency function.
\subsection{Nodal set in a ball}

In order to get an upper bound of the measure of the nodal set of an eigenfunction $u$ in some small ball, we need to estimate an upper bound of the $L^{\infty}$ norm of $u$.

\begin{lemma}\label{upper bound for u}
Let $u$ and $v$ satisfy
\begin{eqnarray*}
\triangle u+\lambda u&=&v,\\
\triangle v-\lambda v&=&0,
\end{eqnarray*}
in $B_1(0)$, and assume $\lambda>1$. Then, for any $r\in(0,1)$ and $t>1$ and $tr<1$, we have
\begin{equation}
\|u\|_{L^{\infty}(B_r(0))}\leq
C\frac{\lambda^{\frac{n+2}{4}}}{r^{\frac{n}{2}+1}}(\|u\|_{L^2(B_{tr}(0))}+\|v\|_{L^2(B_{tr}(0))}),
\end{equation}
where $C$ is a positive constant depending only on $n$ and $t$.
\end{lemma}

\begin{proof}
From the equation $$\triangle u+\lambda u=v,$$ the standard interior estimates of elliptic equations, we have that,
\begin{eqnarray*}
\|u\|_{W^{k,2}(B_r(0))}
&\leq&C\left(\frac{\sqrt{\lambda}}{r}\|u\|_{W^{k-1,2}B_{\kappa r}(0)}+\|v\|_{W^{k-1,2}B_{\kappa r}(0)}\right)
\end{eqnarray*}
for any positive integer $k>0$, where $\kappa>1$  and $C$ is a positive constant depending only on $n$ and $\kappa$.
On the other hand, from the standard elliptic estimation and $\triangle v=\lambda v$, it also holds that
\begin{equation*}
\|v\|_{W^{k,2}(B_{r}(0))}\leq C\left(\frac{\sqrt{\lambda}}{r}\right)^{k}\|v\|_{L^2(B_{\kappa^kr}(0))},
\end{equation*}
where $C>0$ depends only on $n$ and $\kappa$.
From the above inequalities, we have that
\begin{eqnarray*}
\|u\|_{W^{k,2}(B_r(0))}
&\leq&C\left(\frac{\sqrt{\lambda}}{r}\|u\|_{W^{k-1,2}(B_{\kappa r}(0))}+\|v\|_{W^{k-1,2}(B_{\kappa r}(0))}\right)
\\&\leq&C\left(\left(\frac{\sqrt{\lambda}}{r}\right)^2\|u\|_{W^{k-2,2}(B_{\kappa^2 r}(0))}+\frac{\sqrt{\lambda}}{r}\|v\|_{W^{k-2}(B_{\kappa^2 r}(0))}+\|v\|_{W^{k-1,2}(B_{\kappa r}(0))}\right)
\\&\leq&C\left(\left(\frac{\sqrt{\lambda}}{r}\right)^k\|u\|_{L^2(B_{\kappa^kr}(0))}
+\sum\limits_{j=0}^{k-1}\left(\frac{\sqrt{\lambda}}{r}\right)^j\|v\|_{W^{k-1-j,2}(B_{\kappa^jr}(0))}\right)
\\&\leq&C\left(\left(\frac{\sqrt{\lambda}}{r}\right)^{k}\|u\|_{L^2(B_{\kappa^kr}(0))}
+k\left(\frac{\sqrt{\lambda}}{r}\right)^{k-1}\|v\|_{L^2(B_{\kappa^kr}(0))}\right)
\\&\leq&Ck\left(\frac{\sqrt{\lambda}}{r}\right)^k\left(\|u\|_{L^2(B_{\kappa^kr}(0))+\|v\|_{L^2(B_{\kappa^k}(0))}}\right).
\end{eqnarray*}
Then from the Sobolve's imbedding theorem and let $t=\kappa^k$ for $k=[n/2]+1$, we can get the desired result.
\end{proof}

\begin{remark}\label{remark1}
From the relationship of $u$ and $g$, and the doubling condition for $g$,
it is easy to check that for any $0<r_1<r_2<\sqrt{2}/2$, it holds that
\begin{equation}\label{doubling condition from g to u}
\int_{B_{r_2}(x_0)}u^2dx\leq\left(\frac{\sqrt{2}r_2}{r_1}\right)
^{C\max\left\{N(y_0,\sqrt{2}r_2),C_0\right\}+\sqrt{\lambda}r_2-\ln(r_2-r_1)}
\int_{B_{r_1}(x_0)}u^2dx,
\end{equation}
where $y_0=(x_0,0)$ and $C$ is a positive constant depending only on $n$. That is because $B_{r_2}(x_0)\times(-r_2,r_2)\subseteq B_{\sqrt{2}r_2}(y_0)$,
$B_{r_1}(y_0)\subseteq B_{r_1}(x_0)\times(-r_1,r_1)$, and the following estimates:
\begin{eqnarray*}
\int_{B_{r_2}(x_0)}u^2dx&=&\frac{2\sqrt{\lambda}
e^{2\sqrt{\lambda}r_2}}{e^{4\sqrt{\lambda}r_2}-1}\int_{B_{r_2}(x_0)\times(-r_2,r_2)}g^2dy
\\&\leq&\frac{2\sqrt{\lambda}e^{2\sqrt{\lambda}r_2}}{e^{4\sqrt{\lambda}r_2}-1}\int_{B_{\sqrt{2}r_2}(y_0)}g^2dy
\\&\leq&\frac{2\sqrt{\lambda}e^{2\sqrt{\lambda}r_2}}{e^{4\sqrt{\lambda}r_2}-1}
\left(\frac{\sqrt{2}r_2}{r_1}\right)^{C\max\left\{N(y_0,\sqrt{2}r_2),C_0\right\}}
\left(\lambda^2+\frac{1}{(\sqrt{2}r_2-r_1)^4}\right)\int_{B_{r_1}(y_0)}g^2dy
\\&\leq&\frac{2\sqrt{\lambda}e^{2\sqrt{\lambda}r_2}}{e^{4\sqrt{\lambda}r_2}-1}
\left(\frac{\sqrt{2}r_2}{r_1}\right)^{C\max\left\{N(y_0,\sqrt{2}r_2),C_0\right\}}
\left(\lambda^2+\frac{1}{(\sqrt{2}r_2-r_1)^4}\right)
\frac{e^{2\sqrt{\lambda}r_1}-1}{2\sqrt{\lambda}e^{2\sqrt{\lambda}r_1}}
\int_{B_{r_1}(x_0)}u^2dx
\\&\leq&\frac{e^{2\sqrt{\lambda}(r_2+r_1)}}{4\sqrt{\lambda}r_2}
\left(\frac{\sqrt{2}r_2}{r_1}\right)^{C\max\left\{N(y_0,\sqrt{2}r_2),C_0\right\}}
\left(\lambda^2+\frac{1}{(r_2-r_1)^4}\right)\int_{B_{r_1}(x_0)}u^2dx
\\
&\leq&\left(\frac{\sqrt{2}r_2}{r_1}\right)
^{C\left(\max\left\{N(y_0,\sqrt{2}r_2),C_0\right\}+\sqrt{\lambda}r_2+\ln\lambda-\ln(r_2-r_1)\right)}
\int_{B_{r_1}(x_0)}u^2dx.
\end{eqnarray*}
\end{remark}

Basing on the monotonicity formula and the doubling conditions in Remark $\ref{remark1}$, we can get an upper bound for the Hausdorff measure of the nodal set of $u$ in some small ball.

\begin{theorem}\label{measure estimate of nodal set of the case L=tirangle}
Let $\lambda>1$ and $u$ be an eigenfunction of $B_{r_0}(0)$ with $\lambda^2$ the corresponding eigenvalue and $r_0<1$. Then we have the following estimate of the $(n-1)$ dimensional Hausdorff measure of the nodal set of $u$ in $B_{\frac{1}{16}r}(0)$ with a fixed $0<r<r_0/4$:
\begin{equation}\label{measure estimate of nodal set}
\mathcal{H}^{n-1}\left(\left\{x:u(x)=0\right\}\cap B_{\frac{1}{16}r}(0)\right)
\leq C\left(\max\left\{N(0,r_0),C_0\right\}+\ln\lambda+\sqrt{\lambda}r-\ln r\right)r^{n-1},
\end{equation}
where  $C$ is a positive constant depending only on $n$.
\end{theorem}

\begin{proof}
Without loss of generality, we may assume that
$$
\fint_{B_{\sqrt{2}r/2}(0)}u^2dx=1.
$$
Then from Remark $\ref{remark1}$ and Theorem $\ref{frequency change center}$,
and note that for any $p\in\overline{B_{r/4}(0)}$,
$$\overline{B_{r/16}(p)}\subseteq \overline{B_{5r/16}(0)}\subseteq \overline{B_{\sqrt{2}r/2}(0)},\quad
\overline{B_{r/16}(0)}\subseteq \overline{B_{5r/16}(p)},$$
it holds that,
\begin{eqnarray*}
\fint_{B_{\frac{r}{16}}(p)}u^2dx&\geq&2^{-C\left(\max\left\{N((p,0),5\sqrt{2}r/16),C_0\right\}+\sqrt{\lambda}r-\ln r\right)}
\fint_{B_{\frac{5r}{16}}(p)}u^2dx\\&\geq&2^{-C\left(\max\left\{N((p,0),\sqrt{2}r/16),C_0\right\}+\sqrt{\lambda}r-\ln r\right)}
\fint_{B_{\frac{r}{16}}(0)}u^2dx\\&\geq&2^{-C\left(\max\left\{N((0,0),\frac{5\sqrt{2}+2}{8}r),C_0\right\}
+\sqrt{\lambda}r-\ln r\right)}
\fint_{B_{\frac{r}{16}}(0)}u^2dx\\&\geq&2^{-C\left(\max\left\{N((0,0),\frac{5\sqrt{2}+2}{8}r),C_0\right\}
+\sqrt{\lambda}r-\ln r\right)}2^{-C\left(\max\left\{N((0,0),r),C_0\right\}+\sqrt{\lambda}r-\ln r\right)}\fint_{B_{\frac{\sqrt{2}r}{2}}(0)}u^2dx
\\&\geq&
2^{-C(\max\left\{N(0,r_0),C_0\right\}+\sqrt{\lambda}r-\ln r)}.
\end{eqnarray*}
Here we have used the doubling condition in the first and fourth inequalities, and the ``changing center property'' for the frequency function in the third inequality.

The above inequality shows that, for any $p\in\overline{B_{r/4}(0)}$, there exists some point $x_p\in \overline{B_{r/16}(p)}$, such that
\begin{equation}\label{lower bound for upx}
|u(x_p)|\geq2^{-C(\max\left\{N(0,r_0),C_0\right\}+\sqrt{\lambda}r-\ln r)}.
\end{equation}
Choose $p_i\in\partial B_{r/4}(0)$, $i=1,2,\cdots, n$ on the $i-$th axis. Then from $(\ref{lower bound for upx})$ we get that there exist points $x_{p_i}\in B_{\frac{r}{16}}(p_i)$, $i=1,2,\cdots, n$ satisfy $(\ref{lower bound for upx})$. On the other hand, from Lemma $\ref{upper bound for u}$ and Lemma $\ref{h^2 controled by g^2}$, we have
\begin{eqnarray*}
\|u\|_{L^{\infty}(B_{r}(0))}
&\leq&C\frac{\lambda^{\frac{n+2}{4}}}{r^{\frac{n}{2}+1}}(\|u\|_{L^2(B_{5r/4}(0))}+\|v\|_{L^2(B_{5r/4}(0))})
\\&\leq&C\frac{\lambda^{\frac{n+2}{4}}}{r^{\frac{n}{2}+1}}(\|u\|_{L^2(B_{5r/4}(0))}
+2^{C\sqrt{\lambda}r}\|h\|_{L^2(B_{5\sqrt{2}r/4}(0))})
\\&\leq&C\frac{\lambda^{\frac{n+2}{4}}}{r^{\frac{n}{2}+1}}
(\|u\|_{L^2(B_{5r/4}(0))}+2^{C(\sqrt{\lambda}r+\ln\lambda-\ln r)}\|g\|_{L^2(B_{2r}(0))})
\\&\leq&C\frac{\lambda^{\frac{n+2}{4}}}{r^{\frac{n}{2}+1}}
(\|u\|_{L^2(B_{5r/4}(0))}+2^{C(\sqrt{\lambda}r+\ln\lambda-\ln r)}\|u\|_{L^2(B_{2\sqrt{2}r}(0))})
\\&\leq&2^{C(\max\left\{N(0,4r),C_0\right\}+\sqrt{\lambda}r+\ln\lambda-\ln r)}\left(\fint_{B_{\sqrt{2}r/2}(0)}u^2dx\right)^{\frac{1}{2}}.
\end{eqnarray*}
In the above estimates, we have used $\lambda>1$, the assumption that $4r<r_0$ and the monotonicity formula for the frequency function.
Thus we arrive at
\begin{equation}\label{upper bound for u in B1/2}
\|u\|_{L^{\infty}(B_{r}(0))}\leq 2^{C(\max\left\{N(0,r_0),C_0\right\}+\ln\lambda+\sqrt{\lambda}r-\ln r)}.
\end{equation}
Let
\begin{equation*}
f_i(\omega,t)=u(x_{p_i}+t\omega),\quad \omega\in\mathbb{S}^{n-1},\quad t\in(-\frac{5}{8}r,\frac{5}{8}r),\quad i=1,2,\cdots, n,
\end{equation*}
where $\mathbb{S}^{n-1}$ means the unit sphere on $\mathbb{R}^n$. Then $x_{p_i}+t\omega\subseteq B_r(0)$. Then from
$(\ref{upper bound for u in B1/2})$,
we have that
$$
|f_i(\omega,t)|\leq\|u\|_{L^{\infty}(B_{r}(0))}\leq 2^{C(\max\left\{N(0,r_0),C_0\right\}+\ln\lambda+\sqrt{\lambda}r-\ln r)}.
$$
On the other hand,
from $(\ref{lower bound for upx})$, we have
$$
|f_i(\omega,0)|=|u(x_{p_i})|\geq2^{-C(\max\left\{N(0,r_0),C_0\right\}+\sqrt{\lambda}r-\ln r)}.
$$
Then from Lemma $2.3.2$ in $\cite{Q.Han and F.H.Lin}$,
we have that
\begin{equation}\label{upper bound of nodal set on one dimension}
\mathcal{H}^0\left(\left\{t\in(-\frac{5}{8}r,\frac{5}{8}r):u(x_{p_j}+t\omega)=0\right\}\right)\leq C\left(\left\{N(0,r_0),C_0\right\}+\ln\lambda+\sqrt{\lambda}r-\ln r\right).
\end{equation}
Thus from the integral geometric formula (see $\cite{F.H.Lin and X.P.Yang}$, $\cite{Q.Han and F.H.Lin}$), and the
fact that $$B_{r/16}(0)\subseteq\cap_{i=i}^nB_{5r/8}(x_{p_i}),$$ we have
\begin{equation}
\mathcal{H}^{n-1}\left(\left\{x\in B_{\frac{1}{16}r}(0):u(x)=0\right\}\right)
\leq C\left(\left\{N(0,r_0),C_0\right\}+\ln\lambda+\sqrt{\lambda}r-\ln r\right)r^{n-1},
\end{equation}
which is the desired result.
\end{proof}

\subsection{Upper bound for the frequency function}

We first give the definition of the ``doubling index'', which is borrowed from $\cite{A.Logunov and E.malinnikova}$ and $\cite{A.Logunov}$  and will be used to help us to give an upper bound for the frequency function.
\begin{definition}\label{definition of doubling index}
For $B_r(y_0)\subseteq\Omega\times\mathbb{R}$, define
\begin{equation}\label{another definition of frequency function}
\bar{N}(y_0,r)=\log_2\frac{\max\limits_{B_r(y_0)}|g|}{\max\limits_{B_{\frac{r}{2}}(y_0)}|g|}.
\end{equation}
\end{definition}

We now give the relationship between the frequency function and the doubling index.

\begin{lemma}\label{relation of two kind frequency}
Let $y_0=(x_0,0)$, $B_{4r}(x_0)\subseteq\Omega$. Assume that $\lambda>1$ is large enough such that $\log_2\lambda\geq C_0$ and $C_0$ is the constant in Theorem $\ref{monotonicity formula}$. Then it holds that
\begin{equation}
\bar{N}(y_0,r)\leq C\left(\ln\lambda-\ln r+N(y_0,2r)\right),
\end{equation}
and
\begin{equation}
N(y_0,\frac{r}{2})\leq C'\left(\ln\lambda-\ln r+\bar{N}(y_0,r)\right),
\end{equation}
where $C$ and $C'$ are positive constants depending only on $n$. Here $$g(y)=g(x,x_{n+1})=u(x)e^{\sqrt{\lambda}x_{n+1}}$$
is defined in Section 2.
\end{lemma}

\begin{proof}
Note that $\triangle g=h+2\lambda g$ and $\triangle h=0$. So from the standard interior estimate of $g$ and $h$, and note that $g$ is a function defined in a subset of $\mathbb{R}^{n+1}$, we have
\begin{equation*}
\|g\|_{L^{\infty}(B_r(y_0))}\leq C\frac{\lambda^{\frac{n+3}{4}}}{r^{\frac{n+1}{2}+1}}\left(\|g\|_{L^2(B_{2r}(y_0))}+\|h\|_{L^2(B_{2r}(y_0))}\right).
\end{equation*}
From Lemma $\ref{h^2 controled by g^2}$,
\begin{eqnarray*}
\|h\|_{L^2(B_r(x_0))}&\leq&C\left(\frac{1}{r^2}+\lambda\right)\|g\|_{L^2(B_{2r}(x_0))}.
\end{eqnarray*}
So
\begin{equation*}
\|g\|_{L^2(B_r(x_0))}+\|h\|_{L^2(B_r(x_0))}\leq C\left(\frac{1}{r^2}+\lambda\right)\|g\|_{L^{\infty}(B_{2r}(x_0))}r^{\frac{n+1}{2}},
\end{equation*}
i.e.,
\begin{equation*}
\|g\|_{L^{\infty}(B_{2r}(x_0))}\geq
C\frac{1}{\left(\frac{1}{r^2}+\lambda\right)}r^{-\frac{n+1}{2}}(\|g\|_{L^2(B_r(x_0))}+\|h\|_{L^2(B_r(x_0))}).
\end{equation*}
Thus we have
\begin{eqnarray*}
\bar{N}(y_0,r)&=&\log_2\frac{\max\limits_{B_r(y_0)}|g|}{\max\limits_{B_{\frac{r}{2}}(y_0)}|g|}
\\&\leq&\log_2C\frac{\frac{\lambda^{\frac{n+3}{4}}}
{r^{\frac{n+1}{2}+1}}\left(\int_{B_{2r}(y_0)}(g^2+h^2)dy\right)^{\frac{1}{2}}}
{\frac{1}{r^{\frac{n+1}{2}}\left(\lambda+\frac{1}{r^2}\right)}
\left(\int_{B_{\frac{r}{4}}(y_0)}(g^2+h^2)dy\right)^{\frac{1}{2}}}
\\&\leq&C(\ln\lambda-\ln r)+
\frac{1}{2}\log_2\frac{\int_{B_{2r}(y_0)}
(g^2+h^2)dy}{\int_{B_{\frac{r}{4}}(y_0)}(g^2+h^2)dy},
\end{eqnarray*}
where $C$ is a suitable large positive constant depending only on $n$, $\Omega$ and $B_j$, $j=1,2$.
From the doubling condition, i.e., Lemma $\ref{doubling condition for both u and v}$, we have
\begin{equation*}
\log_2\frac{\int_{B_{2r}(y_0)}(g^2+h^2)dy}{\int_{B_{\frac{r}{4}}(y_0)}(g^2+h^2)dy}\leq C\max\left\{N(y_0,2r),C_0\right\}.
\end{equation*}
It tells us that
\begin{equation*}
\bar{N}(y_0,r)\leq C(\ln\lambda-\ln r+\max\left\{N(y_0,2r),C_0\right\}),
\end{equation*}
where $C$ is a positive constant depending only on $n$.
By the similar arguments, we also have that
\begin{eqnarray*}
\bar{N}(y_0,r)&=&\log_2\frac{\max\limits_{B_r(y_0)}|g|}{\max\limits_{B_{\frac{r}{2}}(y_0)}|g|}
\\&\geq&\log_2C\frac{\frac{1}{r^{\frac{n+1}{2}}\left(\lambda+\frac{1}{r^2}\right)}
\left(\int_{B_{\frac{7r}{8}}(y_0)}(g^2+h^2)dy\right)^{\frac{1}{2}}}
{\frac{\lambda^{\frac{n+3}{4}}}{r^{\frac{n+1}{2}+1}}\left(\int_{B_{\frac{5r}{8}}(y_0)}(g^2+h^2)dy\right)^{\frac{1}{2}}}
\\&\geq&C(\ln r-\ln\lambda)+\frac{1}{2}\log_2
\frac{\int_{B_{\frac{7r}{8}}(y_0)}(g^2+h^2)dy}{\int_{B_{\frac{5r}{8}}(y_0)}(g^2+h^2)dy}.
\end{eqnarray*}
Then also from Lemma $\ref{doubling condition for both u and v}$, it holds that
\begin{equation*}
\bar{N}(g,B_r(y_0))\geq C(\ln r-\ln \lambda+\min_{\rho\in[5r/8,7r/8]}N(y_0,\rho)).
\end{equation*}
If for any $\rho\in[5r/8,7r/8]$,  $N(y_0,\rho)\geq C_0$, then from the monotonicity formula, i.e., Theorem $\ref{monotonicity formula}$, we have
\begin{equation*}
\bar{N}(g,B_r(y_0))\geq C(\ln r-\ln\lambda+N(y_0,r/2)),
\end{equation*}
and this implies that
\begin{equation*}
N(y_0,r/2)\leq C(\bar{N}(y_0,r)+\ln\lambda-\ln r).
\end{equation*}
If for some $\rho_0\in[5r/8,7r/8]$, $N(y_0,\rho_0)\leq C_0$,
then also from Theorem $\ref{monotonicity formula}$, we have
\begin{equation*}
N(y_0,r/2)\leq C\max\left\{N(y_0,\rho_0),C_0\right\}\leq CC_0\leq C(\bar{N}(y_0,r)+\ln\lambda-\ln r),
\end{equation*}
if $\lambda$ is large enough. Thus we can get the result we need.
\end{proof}

Now we will go to establish the upper bound for the frequency function. We first prove the following two lemmas.

\begin{lemma}\label{neighborhood of the nonanalytic point}
Let $\Omega$ be a bounded domain in $\mathbb{R}^n$. Assume that $\partial\Omega$ is of $C^{\infty}$.
Let $u$ be a solution of the eigenvalue problem $\triangle^2u=\lambda^2u$ in $\Omega$ with the boundary conditions $(\ref{boundary conditions})$ and define
$T_r(\partial\Omega)=\left\{x\in\overline{\Omega}:dist(x,\partial\Omega)\leq r\right\}$. Then we have
\begin{equation}
\|u\|_{L^2(T_{r^*}(\partial\Omega))}\leq\frac{1}{2}\|u\|_{L^2(\Omega)},
\end{equation}
where $r^*=C_1\lambda^{-(n+2)/2}<1$, and $C_1$ is a positive constant depending only on $n$, $\Omega$ and the boundary operators $B_j$, $j=1,2$.
\end{lemma}

\begin{proof}
First we have the following global $L^{\infty}(\Omega)$ estimate of $u$:
\begin{equation}\label{standard elliptic estiamte}
\|u\|_{L^{\infty}(\Omega)}\leq C\lambda^{\frac{n+2}{4}}\|u\|_{L^2(\Omega)},
\end{equation}
where $C$ is a positive constant depending only on $n$, $\Omega$ and the boundary operators $B_j$,
$j=1,2$.
Thus
\begin{eqnarray*}
\|u\|_{L^2(T_r(\partial\Omega))}&\leq&C\|u\|_{L^{\infty}(\Omega)}(\mathcal{H}^n(T_r(\partial\Omega)))^{\frac{1}{2}}
\\&\leq&C\lambda^{\frac{n+2}{4}}\|u\|_{L^2(\Omega)}((\mathcal{H}^{n-1}(\partial\Omega)r)^{\frac{1}{2}})
\\&\leq&C\lambda^{\frac{n+2}{4}}\|u\|_{L^2(\Omega)}r^{\frac{1}{2}}
\\&\leq&\frac{1}{2}\|u\|_{L^2(\Omega)},
\end{eqnarray*}
if $0<r\leq C_1\lambda^{-(n+2)/2}$ for some suitable constant $C_1$ depending only on $n$, $\Omega$, $\mathcal{H}^{n-1}(\partial\Omega)$, and the boundary operators $B_j$, $j=1,2$. That is the desired result.
\end{proof}

Because $\partial\Omega$ is analytic except the set $\Gamma$, one can extend the function $u$ out of $\Omega$ except a neighborhood of $\Gamma$. Thus we have the following conclusion.

\begin{lemma}\label{extending}
Let $T_{r}(\Gamma)=\left\{x\in\overline{\Omega}:dist(x,\Gamma)\leq r\right\}$ be the $r$ tubular type domain containing  $\Gamma$.
Let $x\in\overline{\Omega}\setminus T_{\widetilde{r}}(\Gamma)$, where $\widetilde{r}=\lambda^{-d}<1$ for some positive constant $d$. Then for any $\tau\in(0,1)$ there exists a positive constant $C$ depending only on $n$, $\Omega$, $d$, and $B_j$, $j=1,2$, such that for any $1\geq r\geq\widetilde{r}$, $u$ can be analytically
extended into the set $B_{\tau r}(x)\setminus\Omega$, and
\begin{equation}
\|u\|_{L^{\infty}(B_{\tau r}(x))}\leq e^{C\sqrt{\lambda}}\|u\|_{L^2(B_{r}(x)\cap\Omega)}.
\end{equation}

\end{lemma}

\begin{proof}
Because $x\in\overline{\Omega}\setminus T_{r}(\Gamma)$ with $r\geq\widetilde{r}$, we know that $dist(x,\Gamma)>r$. So
from the standard elliptic estimate, the Sobolev's embedding theorem and the fact that $\partial\Omega$ is of $C^{\infty}$ and compact, we know that there exists a positive constant $C$ depending only on $n$, $\Omega$, $B_j$, $j=1,2$, but independent of $x$, such that for any fixed multi-index $\alpha$,
\begin{equation*}
|D^{\alpha}u(x)|\leq C\frac{\lambda^{(\frac{|\alpha|}{2}+\frac{n+2}{4})}}
{r^{|\alpha|+1+\frac{n}{2}}}\|u\|_{L^2(B_{r}(x)\cap\Omega)}.
\end{equation*}
Because $u$ is analytic in some neighborhood of $x$, the Taylor power series of $u$ at point $x$ is convergent in $B_{\tau r}(x)$ for any $\tau\in(0,1)$. So
\begin{eqnarray*}
\|u\|_{L^{\infty}(B_{\tau r}(x))}&\leq&\sum\limits_{k=0}^{\infty}\sum\limits_{|\alpha|=k}
\frac{1}{k!}|D^{\alpha}u(x)|(\tau r)^{k}\\&\leq&
C\sum\limits_{k=0}^{\infty}\frac{k^n}{k!}(\tau\sqrt{\lambda})^k\frac{\lambda^{\frac{n+2}{4}}}{(r)^{\frac{n}{2}+1}}
\|u\|_{L^2(B_{r}(x)\cap\Omega)}
\\&\leq&e^{C(\tau\sqrt{\lambda}+\ln\lambda-\ln r)}\|u\|_{L^2(B_{r}(x)\cap\Omega)}
\\&\leq&e^{C\sqrt{\lambda}}\|u\|_{L^2(B_{r}(x)\cap\Omega)}.
\end{eqnarray*}
In the last inequality in the above, we have used the fact that $r\geq\widetilde{r}=\lambda^{-d}$, $\ln\lambda<\sqrt{\lambda}$ for $\lambda$ large enough, and
$C$ is a positive constant depending only on $n$, $\Omega$, $d$ and $B_j$, $j=1,2$, but independent of $x$ because $\partial\Omega$ is of $C^{\infty}$ and compact.
That is the desired result.
\end{proof}

Now we begin to show an upper bound for the frequency function.

\begin{theorem}\label{upper bound for the frequency function}
Let $u$ be an eigenfunction in $\Omega$ with the boundary condition $(\ref{boundary conditions})$.
Then, there exists a positive constant $R_0$ depending only on $n$ and $\Omega$, such that for any $x\in\Omega(R_0,\Gamma)=\overline{\Omega\setminus T_{R_0}(\Gamma)}$ and $0<r<R_0/2$, it holds that
\begin{equation}
N(y,r)\leq C\sqrt{\lambda}.
\end{equation}
Here $y=(x,0)$ and the positive constants $C$ depends only on $n$, $\Omega$ and $B_j$, $j=1,2$.
\end{theorem}

\begin{proof}
Without loss of generality, we may assume that $\|u\|_{L^2(\Omega)}=1$. Then from Lemma $\ref{neighborhood of the nonanalytic point}$, we know that $\|u\|_{L^2(\Omega\setminus T_{r^*})(\partial\Omega)}\geq\frac{1}{2}$ with $r^*=C_1\lambda^{-(n+2)/2}$ is the same radius as in Lemma $\ref{neighborhood of the nonanalytic point}$.
Moreover, the function $u$ is analytic at any point of the set
$\Omega(r^*,\partial\Omega)=\overline{\Omega\setminus T_{r^*}(\partial\Omega)}$.

Let $\bar{x}$ denote a maximum point of $|u|$ in $\Omega(r^*,\partial\Omega)$ and define $y_0=(\bar{x},0)$.
From the $L^{\infty}$ estimation of $u$, we have
\begin{equation*}
\|u\|_{L^{\infty}(B_{r^*}(\bar{x}))}\leq \|u\|_{L^{\infty}(\Omega)}\leq C\lambda^{\frac{n+2}{4}}\|u\|_{L^2(\Omega)}\leq e^{C\ln\lambda}.
\end{equation*}
On the other hand,
\begin{equation*}
|u(\bar{x})|=\|u\|_{L^{\infty}(\Omega(r^*,\partial\Omega))}\geq \frac{1}{\left(\mathcal{H}^n(\Omega)\right)^{\frac{1}{2}}}\|u\|_{L^2(\Omega(r^*,\partial\Omega))}
\geq\frac{1}{2\left(\mathcal{H}^n(\Omega)\right)^{\frac{1}{2}}}.
\end{equation*}
Then the doubling index of $g$ centered at $y_0$ with radius $r_0=r^*$ can be bounded as follows:
\begin{eqnarray*}
\bar{N}(y_0,r_0)&=&\log_2\frac{\max\limits_{y\in B_{r_0}(y_0)}|g(y)|}{\max\limits_{y\in B_{r_0/2}(y_0)}|g(y)|}
\\&\leq&\log_2\frac{\max\limits_{y\in B_{r_0}(y_0)}|g(y)|}{|g(y_0)|}
\\&\leq&\log_2\frac{e^{\sqrt{\lambda}r_0}\max\limits_{x\in B_{r_0}(\bar{x})}|u(x)|}{|u(\bar{x})|}
\\&\leq&C\sqrt{\lambda}r_0+\log_2\frac{e^{C\ln\lambda}}{C}
\\&\leq&C(\sqrt{\lambda}r_0+\ln\lambda).
\end{eqnarray*}

From Lemma $\ref{relation of two kind frequency}$, we have
\begin{equation*}
N(y_0,\frac{r_0}{2})\leq C(\sqrt{\lambda}r_0+\ln\lambda-\ln r_0).
\end{equation*}
Thus from Theorem $\ref{frequency change center}$, for any point $p\in B_{\frac{r_0}{16}}(y_0)$,
\begin{equation*}
N(p,\frac{r_0}{8})\leq C(\sqrt{\lambda}r_0+\ln\lambda-\ln r_0),
\end{equation*}
and  thus from Lemma $\ref{relation of two kind frequency}$ again,
\begin{equation*}
\bar{N}(p,\frac{r_0}{16})\leq C(\sqrt{\lambda}r_0+\ln\lambda-\ln r_0).
\end{equation*}
Then from the fact that $y_0\in B_{r_0/16}(p)$, we have
\begin{eqnarray*}
\max\limits_{B_{\frac{r_0}{32}}(p)}|g|&\geq&\max\limits_{B_{\frac{r_0}{16}}(p)}|g|
e^{-C(\sqrt{\lambda}r_0+\ln\lambda-\ln r_0)}
\\&\geq&|u(\bar{x})|e^{-C(\sqrt{\lambda}r_0+\ln \lambda-\ln r_0)}
\\&\geq&e^{-C(\sqrt{\lambda}r_0+\ln \lambda-\ln r_0)}.
\end{eqnarray*}
Choose $x_1\in\overline{B_{r_0/16}(\bar{x})}$ be the point such that
$$dist(x_1,\partial\Omega)=\max_{x\in B_{r_0/16}(\bar{x})\cap\Omega(r_0,\partial\Omega)}dist(x,\partial\Omega),$$
and let $y_1=(x_1,0)\in B_{r_0/16}(y_0)$. Then from the above inequalities, we have
\begin{equation*}
\max\limits_{B_{\frac{r_0}{2}+\frac{r_0}{32}}(y_1)}|g|\geq\max\limits_{B_{\frac{r_0}{32}}(y_1)}|g|\geq e^{-C(\sqrt{\lambda}r_0+\ln\lambda-\ln r_0)}.
\end{equation*}
Let $r_1=r_0+\frac{r_0}{16}$, then $dist(x_1,\partial\Omega)\geq r_0+r_0/16=r_1$. Thus $B_{r_1}(x_1)\subseteq\Omega$.
From the $L^{\infty}$ estimation of $u$ in $\Omega$ again, we have
\begin{equation*}
\|u\|_{L^{\infty}(B_{r_1}(x_1))}\leq e^{C\ln\lambda}.
\end{equation*}
Thus it holds that
\begin{eqnarray*}
\bar{N}(y_1,r_1)&=&\log_2\frac{\max\limits_{y\in B_{r_1}(y_1)}|g(y)|}{\max\limits_{y\in B_{r_1/2}(y_1)}|g(y)|}
\\&\leq&\log_2\frac{e^{\sqrt{\lambda}r_1}\max\limits_{x\in B_{r_1}(x_1)}|u(x)|}{e^{-C(\sqrt{\lambda}r_0+\ln\lambda-\ln r_0)}}
\\&\leq&\sqrt{\lambda}r_1+\log_2\frac{e^{C\ln\lambda}}
{e^{-C(\sqrt{\lambda}r_0+\ln\lambda-\ln r_0)}}
\\&\leq&C(\sqrt{\lambda}(r_0+r_1)+\ln\lambda-\ln r_0).
\end{eqnarray*}
Then from Lemma $\ref{relation of two kind frequency}$, we have
\begin{equation*}
N(y_1,r_1/2)\leq C(\sqrt{\lambda}(r_0+r_1)+\ln\lambda-\ln r_0-\ln r_1).
\end{equation*}
Applying the same argument with replacing $r_0$ by $r_1$, we have for $r_2=r_1+\frac{r_1}{16}$, $y_2=(x_2,0)$, where $x_2\in\Omega(r_0,\partial\Omega)$ and $$dist(x_2,\partial\Omega)=\max_{x\in      B_{r_1/16}(x_1)\cap\Omega(r_0,\partial\Omega)}dist(x,\partial\Omega),$$
it holds that
\begin{equation*}
\bar{N}(y_2,r_2)\leq C(\ln\lambda-\ln r_0-\ln r_1+\sqrt{\lambda}(r_0+r_1+r_2)),
\end{equation*}
and thus
\begin{equation*}
N(y_2,r_2/2)\leq C(2\ln\lambda-\ln r_0-\ln r_1-\ln r_2+\sqrt{\lambda}(r_0+r_1+r_2)).
\end{equation*}
Let $\delta=\frac{17}{16}$.
 Choose $R_0$ to be a suitable small positive constant depending only on n and $\Omega$ but independent of $\lambda$, such that the following conditions hold:

(1) $R_0$ satisfies that $\mathcal{H}^n(T_{R_0}(\partial\Omega))\leq\frac{1}{100}\mathcal{H}^n(\Omega)$;

(2) For any fixed $r\leq R_0$, it holds that, $\forall$ $x\in\partial\Omega(r,\partial\Omega)$, there exists unique point $x'\in\partial\Omega$, such that $dist(x,x')=r$;

Repeat the same argument for $k$ times, such that $r_{k-1}<R_0$, and $r_k\geq R_0$.
Then
\begin{equation*}
R_0\leq r_k=\delta^kr_0=C_1\delta^k\lambda^{-\frac{n+2}{2}},
\end{equation*}
and
\begin{equation*}
R_0\geq r_{k-1}=C_1\delta^{k-1}\lambda^{-\frac{n+2}{2}}.
\end{equation*}
These show that $k\leq\bar{C}\ln\lambda$, where $\bar{C}$ is a positive constant depending only on $R_0$, $C_1$ and $n$. Thus we have, $x_k\in\Omega(R_0,\partial\Omega)$, $y_k=(x_k,0)$, and
\begin{eqnarray*}
\bar{N}(y_k,r_k)&\leq&C(\ln^2\lambda-k\ln r_0-(1+2+\cdots+(k-1))\ln\delta+\sqrt{\lambda}r_0(1+\delta+\delta^2+\cdots+\delta^{k}))
\\&\leq&C(\ln^2\lambda+\sqrt{\lambda}r_0(\delta^{k+1}-1)/(\delta-1))
\\&\leq&C(\ln^2\lambda+\sqrt{\lambda}r_0\delta^{k+1})
\\&\leq&C\sqrt{\lambda},
\end{eqnarray*}
if $\lambda$ is large enough such that $\ln\lambda<\sqrt{\lambda}$.
These imply that
\begin{equation}\label{upper bound for doubling index with radius R0}
\bar{N}(y_k,R_0)\leq C\sqrt{\lambda}.
\end{equation}

Now we will consider the upper bound for the doubling index centered at $y=(x,0)$ for any fixed point $x\in\Omega(R_0,\Gamma)$.
Because $x_k$ and $x$ can be connected by some curve in $\Omega(R_0,\Gamma)$ whose length is bounded by some positive constant depending only on $\Omega$, the upper bound for the doubling index centered at $y$ with radius $R_0$ can be obtained by the similar iteration  arguments as above in finite many steps, and the number of the iteration steps depends only on $\Omega$ and $dist(x_k,x)$. In fact, from Lemma $\ref{extending}$, we know that for $\bar{y}_1=(\bar{x}_1,0)$ with $\bar{x}_1\in\overline B_{\tau R_0/16}(x_k)\cap\Omega(R_0,\Gamma)$, where $\tau\in(0,1)$, it holds that
\begin{eqnarray*}
\|g\|_{L^{\infty}(B_{\tau R_0}(\bar{y}_1))}\leq e^{C\sqrt{\lambda}R_0}\|u\|_{L^{\infty}(B_{\tau R_0}(\bar{x}_1))}\leq e^{C\sqrt{\lambda}}\|u\|_{L^2(B_{R_0}(\bar{x}_1)\cap\Omega)}.
\end{eqnarray*}
On the other hand, because $B_{\tau R_0/4}(y_k)\subseteq B_{\tau R_0/2}(\bar{y}_1)$, we have, from $(\ref{upper bound for doubling index with radius R0})$,
\begin{eqnarray*}
\|g\|_{L^{\infty}(B_{\tau R_0/2}(\bar{y}_1))}&\leq&\|g\|_{L^{\infty}(B_{\tau R_0/4}(y_k))}
\\&\geq& e^{-C\sqrt{\lambda}}\|g\|_{L^{\infty}(B_{R_0}(y_k))}\\&\geq& e^{-C\sqrt{\lambda}}\|u\|_{L^2(B_{R_0}(\bar{x}_1))}
\\&\geq&e^{-C\sqrt{\lambda}}\|u\|_{L^2(B_{R_0}(\bar{x}_1)\cap\Omega)}.
\end{eqnarray*}

Thus from the above inequalities, we have
\begin{equation*}
\bar{N}(\bar{y}_1,\tau R_0)\leq C(\sqrt{\lambda}+\sqrt{\lambda}\tau R_0)\leq C\sqrt{\lambda}.
\end{equation*}

Then by repeating the finite steps whose number depends only on $R_0$ and $\Omega$, we can get the desired result by using Lemma $\ref{relation of two kind frequency}$ and Theorem $\ref{monotonicity formula}$.
\end{proof}

Combining Theorem $\ref{upper bound for the frequency function}$ and Theorem $\ref{measure estimate of nodal set of the case L=tirangle}$, and note that we have already extended $u$ outside $\Omega$ except a neighborhood of $\Gamma$ in Lemma $\ref{extending}$, we can get that

\begin{theorem}\label{nodal set of interior domain}
Let $\Omega(R_0,\Gamma)=\overline{\Omega\setminus T_{R_0}(\Gamma)}$, where $R_0$ is the same positive constant as in Theorem $\ref{upper bound for the frequency function}$. Then
\begin{equation}
\mathcal{H}^{n-1}\left(\left\{x\in\Omega(R_0,\Gamma):u(x)=0\right\}\right)\leq C\sqrt{\lambda}.
\end{equation}
Here $C$ is a positive constant depending only on $n$, $\Omega$ and the boundary operators $B_j$, $j=1,2$.
\end{theorem}


\section{Measure estimates of nodal sets near $\Gamma$}

In this section, we begin to consider the domain including the nonanalytic set $\Gamma$.

\begin{lemma}\label{extending near x0}
Let $x\in\overline{\Omega}\setminus\Gamma$, and $r=dist(x,\Gamma)>0$.
Then the eigenfunction $u$ can be extended analytically into $B_{\sigma r}(x)$ for any $\sigma\in(0,1)$, and satisfies that
\begin{equation}
\|u\|_{L^{\infty}(B_{\sigma r}(x))}\leq e^{C(\sqrt{\lambda}-\ln r)}\|u\|_{L^2(B_r(x)\cap\Omega)},
\end{equation}
where $C$ is a positive constant depending only on $n$, $\Omega$ and $B_j$, $j=1,2$.
\end{lemma}

\begin{proof}
From the proof of Lemma $\ref{extending}$, it is known that, $\forall$ $x\in\overline{\Omega}\setminus\Gamma$ and any $\sigma\in(0,1)$,
\begin{equation*}
\|u\|_{L^{\infty}(B_{\sigma r}(x))}
\leq e^{C(\sqrt{\lambda}+\ln\lambda-\ln r)}\|u\|_{L^2(B_{r}(x)\cap\Omega)},
\end{equation*}
where $C$ is a positive constant depending only on $n$, $\Omega$ and $B_j$, $j=1,2$. Because $\ln\lambda<\sqrt{\lambda}$ when $\lambda$ large enough, the desired result is hold.
\end{proof}

By the same iteration argument in the proof of Theorem $\ref{upper bound for the frequency function}$, we have that

\begin{lemma}\label{frequency function near the point x0}
Let $u$ be an eigenfunction for the bi-harmonic operator with boundary conditions $(\ref{boundary conditions})$.
Let $x\in T_{R_0}(\Gamma)\setminus\Gamma$ and denote $\bar{r}=dist(x,\Gamma)$, where $R_0$ is the same constant as in Theorem $\ref{upper bound for the frequency function}$. Then for any $r\in(0,\bar{r})$, it holds that
\begin{equation}
N(x,r)\leq C(-\sqrt{\lambda}\ln r+\ln^2r).
\end{equation}
Here $C$ is a positive constant depending only on $n$, $\Omega$ and $B_j$, $j=1,2$.
\end{lemma}

\begin{proof}
For any $x^*_0\in\overline{\Omega}$ with $dist(x^*_0,\Gamma)=R_0$, from Theorem $\ref{upper bound for the frequency function}$, we have already known that
\begin{equation*}
\bar{N}(x^*_0, r)\leq C\sqrt{\lambda}
\end{equation*}
for any $r<R_0$. Let $\sigma$ be a constant in $(0,1)$.
Because $\Gamma$ is compact, we know that
\begin{equation*}
T_{R_0}(\Gamma)\setminus T_{(1-\epsilon)R_0}(\Gamma)\subseteq\cup_{x^*\in\overline{\Omega},dist(x^*,\Gamma)=R_0}B_{\sigma R_0/16}(x^*),
\end{equation*}
where $\epsilon=\theta\sigma/16$ and $\theta\in(0,1)$ is a positive constant depending only on $\Omega$ and $\Gamma$. Let $R_1=(1-\epsilon)R_0$.
Then for any $x^*_1\in\overline{\Omega}$ such that $dist(x^*_1,\Gamma)=R_1$, there exists some point $x^*_0\in\overline{\Omega}$ with $dist(x^*_0,\Gamma)=R_0$ such that $x^*_1\in\overline{B_{\sigma R_0/16}(x^*_0)}$.
\begin{eqnarray*}
\|u\|_{L^{\infty}(B_{\sigma R_1/4}(x^*_1))}&\geq&
\|u\|_{L^{\infty}(B_{\sigma R_0/8}(x^*_0))}\\&\geq&
e^{-C\sqrt{\lambda}}\|u\|_{L^{\infty}(B_{\sigma R_0}(x^*_0)\cap\Omega)}.
\end{eqnarray*}
Thus for $y^*_1=(x^*_1,0)$, it holds that
\begin{eqnarray*}
\|g\|_{L^{\infty}(B_{\sigma R_1/2}(y^*_1))}&\geq&e^{-C\sqrt{\lambda}R_1}\|u\|_{L^{\infty}(B_{\sigma R_1/4}(x^*_1))}\\&\geq& e^{-C(\sqrt{\lambda}+\sqrt{\lambda}R_1)}\|u\|_{L^{\infty}(B_{\sigma R_0}(x^*_0)\cap\Omega)}.
\end{eqnarray*}
On the other hand, from Lemma $\ref{extending near x0}$, we have
\begin{eqnarray*}
\|u\|_{L^{\infty}(B_{\sigma R_1}(x^*_1))}&\leq&e^{C(\sqrt{\lambda}-\ln R_1)}\|u\|_{L^{\infty}(B_{R_1}(x^*_1)\cap\Omega)}\\&\leq&
e^{C(\sqrt{\lambda}-\ln R_1)}\|u\|_{L^{\infty}(B_{R_0}(x^*_0)\cap\Omega)}.
\end{eqnarray*}
Thus
\begin{equation*}
\|u\|_{L^{\infty}(B_{\sigma R_1}(y^*_1))}\leq e^{C(\sqrt{\lambda}+\sqrt{\lambda}R_1-\ln R_1)}\|u\|_{L^{\infty}(B_{R_0}(x^*_0)\cap\Omega)}.
\end{equation*}
From the above inequalities and the relationship between $u$ and $g$, we have
\begin{equation*}
\bar{N}(y^*_1,\sigma R_1)\leq C(\sqrt{\lambda}-\ln R_1+\sqrt{\lambda}R_1).
\end{equation*}
Repeat these arguments for $k$ times such that $R_{k-1}\geq r$ and $R_{k}<r$. Then  we have that for $y=(x,0)$,
\begin{eqnarray*}
\bar{N}(y,\sigma\bar{r})&\leq&C\bar{N}(y^*_k,R_k)\\&\leq&C(k\sqrt{\lambda}-(\ln R_1+\cdots+\ln R_k)+\sqrt{\lambda}(R_1+\cdots+R_k))
\\&\leq&C\left(k\sqrt{\lambda}-k\ln R_0-(1+2+\cdots+k)\ln(1-\epsilon)+\sqrt{\lambda}R_0\frac{1-(1-\epsilon)^k}{1-(1-\epsilon)}\right)
\\&\leq&C(k\sqrt{\lambda}-k\ln R_0+k^2+R_0\sqrt{\lambda}).
\end{eqnarray*}
From $R_{k-1}\geq r$ and $R_{k}<r$, we have that
\begin{equation*}
(1-\epsilon)^{k-1}R_0\geq r,
\end{equation*}
and
\begin{equation*}
(1-\epsilon)^kR_0<r.
\end{equation*}
These show that
\begin{equation*}
k\leq -C\ln r
\end{equation*}
for some positive constant $C$ depending only on $n$, $\Omega$, $\Gamma$, and $R_0$.
Thus we have
\begin{eqnarray*}
\bar{N}(x,\sigma r)&\leq&C(k\sqrt{\lambda}-k\ln R_0+k^2+R_0\sqrt{\lambda})
\\&\leq&C(-\ln r\sqrt{\lambda}+\ln^2 r+R_0\sqrt{\lambda})
\\&\leq&C(-\ln r\sqrt{\lambda}+\ln^2r).
\end{eqnarray*}
Because $\sigma$ can be chosen as any number in $(0,1)$, we can get the desired result.
\end{proof}

Now we can establish the measure upper bound of the nodal set of $u$ near the boundary $\partial\Omega$.

\begin{theorem}\label{measure of nodal set near the boundary}
Let $u$ be an eigenfunction of the bi-harmonic operator on a $C^{\infty}$ bounded domain $\Omega\subseteq\mathbb{R}^n$ with the boundary condition $B_ju=0$, $j=1,2$ on $\partial\Omega$, and the corresponding eigenvalue is $\lambda^2$. Also suppose that $\partial\Omega$ is piecewise analytic and $\partial\Omega\setminus\Gamma$ is analytic, where $\Gamma\subseteq\partial\Omega$ is a finite union of some $(n-2)$ dimensional submanifolds of $\partial\Omega$. Then for $\lambda$ large enough, we have the following measure estimate:
\begin{equation}
\mathcal{H}^{n-1}\left(\left\{x\in T_{R_0}(\Gamma)\cap\Omega|u(x)=0\right\}\right)\leq C\sqrt{\lambda},
\end{equation}
where $C$ is a positive constant depending only on $n$, $\Omega$, $B_j$, $j=1,2$ and $R_0$. Here $R_0$ is the same positive constant as in Theorem $\ref{upper bound for the frequency function}$.
\end{theorem}

\begin{proof}

First consider the nodal set in $\left(T_{R_0}(\Gamma)\setminus T_{R_0/2}(\Gamma)\right)\cap\Omega$.
One can use finitely many balls with radius $\sigma R_0$ to cover the set $(T_{R_0}(\Gamma)\setminus T_{R_0/2}(\Gamma))\cap\Omega$. Here $\sigma$ is the same positive constant as in Lemma $\ref{extending near x0}$. The number of these balls is $C/R_0^{n-2}$, where $C>0$ depends only on $n$ and $\Omega$, and is independent of the radius $R_0$.
Because $u$ is analytic in $\left(T_{R_0}(\Gamma)\setminus T_{R_0/2}(\Gamma)\right)\cap\Omega$,
and the upper bound for the frequency function in this case is $C(\sqrt{\lambda}-\ln\lambda\ln R_0-\ln^2R_0)$
from Lemma $\ref{frequency function near the point x0}$, we have
\begin{eqnarray*}
\mathcal{H}^{n-1}\left(\left\{x\in \left(T_{R_0}(\Gamma)\setminus T_{R_0/2}(\Gamma)\right)\cap\Omega|u(x)=0\right\}\right)&\leq& C(-\ln R_0\sqrt{\lambda}+\ln^2R_0)R_0^{n-1}\cdot\frac{1}{R_0^{n-2}}\\&\leq&C(-R_0\ln R_0\sqrt{\lambda}+R_0\ln^2R_0).
\end{eqnarray*}

Now we consider the nodal set in $\left(T_{R_0/2}(\Gamma)\setminus T_{R_0/4}(\Gamma)\right)\cap\Omega$. One can also use finitely many balls with radius $\sigma R_0/2$ to cover the set $\left(T_{R_0/2}(\Gamma)\setminus T_{R_0/4}(\Gamma)\right)\cap\Omega$, and the number of these balls is $C/(R_0/2)^{n-2}$ with $C>0$ depends only on $n$ and $\Omega$. Then we can get the measure of the nodal set in this domain as follows.
\begin{eqnarray*}
\mathcal{H}^{n-1}\left(\left\{x\in (T_{R_0/2}(\Gamma)\setminus T_{R_0/4}(\Gamma))\cap\Omega:u(x)=0\right\}\right)&\leq& C\left(-\ln\frac{R_0}{2}\sqrt{\lambda}+\ln^2\frac{R_0}{2}\right)
\frac{R_0}{2}^{n-1}\cdot\frac{1}{\left(\frac{R_0}{2}\right)^{n-2}}
\\&\leq&C\left(-\frac{R_0}{2}\ln\frac{R_0}{2}\sqrt{\lambda}+\frac{R_0}{2}\ln^2\frac{R_0}{2}\right).
\end{eqnarray*}

Continue this argument step by step.
By the iteration method, we have that
\begin{eqnarray*}
\mathcal{H}^{n-1}\left(\left\{x\in T_{R_0}(\Gamma)\cap\Omega:u(x)=0\right\}\right)&\leq&
\sum\limits_{j=0}^{\infty}\mathcal{H}^{n-1}\left(\left\{x\in B_{R_0/2^j}(\Gamma)\setminus B_{R_0/2^{j+1}}(\Gamma)\cap\Omega|u(x)=0\right\}\right)\\&
\\&\leq&CR_0\left(\sum\limits_{j=0}^{\infty}\frac{j\ln 2-\ln R_0}{2^j}\sqrt{\lambda}+\sum\limits_{j=0}^{\infty}\frac{1}{2^j}\ln^2\frac{R_0}{2^j}\right)
\\&\leq&CR_0\left(\sqrt{\lambda}\sum\limits_{j=0}^{\infty}\frac{j}{2^j}+\sum\limits_{j=0}^{\infty}\frac{j^2}{2^j}\right)
\\&\leq&C\sqrt{\lambda},
\end{eqnarray*}
where the last inequality used the assumption that $\lambda$ large enough. That is the desired result.
\end{proof}

From Theorem $\ref{measure of nodal set near the boundary}$ and Theorem $\ref{nodal set of interior domain}$, we can get the desired upper bound estimate of the nodal set of $u$ in $\Omega$.

\begin{theorem}\label{finally theorem}
Let $u$ be an eigenfunction satisfies the boundary condition $(\ref{boundary conditions})$ and $\lambda^2$ be the corresponding eigenvalue. Assume that $\Omega$ is a $C^{\infty}$ bounded domain, $\partial\Omega$ is piecewise analytic, and $\partial\Omega\setminus\Gamma$ is analytic, where  $\Gamma\subseteq\partial\Omega$ is a finite union of some $(n-2)$ dimensional submanifolds of $\partial\Omega$. Then
\begin{equation}
\mathcal{H}^{n-1}\left\{x\in\Omega|u(x)=0\right\}\leq C\sqrt{\lambda},
\end{equation}
where $C$ is a positive constant depending only on $n$, $\Omega$ and the boundary operators $B_j$, $j=1,2$.
\end{theorem}

\begin{remark}\label{condition weakend}
Note that the $C^{\infty}$ property for $\partial\Omega$ and the coefficients of the boundary operators $B_j$, $j=1,2$, are only used in the $L^{\infty}$ estimation of $u$ and the uniform $L^{\infty}$ estimation of $D^{\alpha}u$ for any fixed multi-index $\alpha$. Thus we may only assume that
$\partial\Omega$ is piecewise analytic and $C^{(n+1)/2}$ continuous, and $\partial\Omega\setminus\Gamma$ is analytic.
\end{remark}

\begin{remark}
If we assume that
\begin{equation}\label{very strong condition}
\|u\|_{L^{\infty}(\Omega)}\leq e^{C\sqrt{\lambda}}\|u\|_{L^2(\Omega)},
\end{equation}
then the smooth condition for the domain $\Omega$ can be weakened further. In fact, if $(\ref{very strong condition})$ holds, then we only need that $\Omega$ is bounded, $\partial\Omega$ is continuous and piecewise  analytic, and $\partial\Omega\setminus\Gamma$ is analytic. Because the $C^{(n+1)/2}$ continuity property for $\partial\Omega$ in Remark $\ref{condition weakend}$ is only used in the estimation of $L^{\infty}$ norm of $u$.
Moreover, if one of the boundary conditions in $(\ref{boundary conditions})$ is $u=0$ on $\partial\Omega$, then the condition $(\ref{very strong condition})$ can also be omitted.
The reason is that, under these assumptions, the upper bounds for the $L^{\infty}$ norms of $u$ and $D^{\alpha}u$ in $\Omega$ are
\begin{equation*}
\|u\|_{L^{\infty}(\Omega)}\leq C\lambda^{\frac{n+1}{4}}\|u\|_{L^2(\Omega)}
\end{equation*}
and
\begin{equation*}
|D^{\alpha}u(x)|\leq C\frac{\lambda^{(\frac{|\alpha|}{2}+\frac{n+2}{4})}}
{r^{|\alpha|+1+\frac{n}{2}}}\|u\|_{L^2(B_{r}(x)\cap\Omega)}
\end{equation*}
respectively, provided that $\partial\Omega$ is bounded and continuous.
This case contains a lot of usual domains. For instance, a polygon in two dimensional case, a polyhedron in three dimensional case, etc..
\end{remark}


\newpage


\end{document}